%% file: Gorenstein_silting_modules_and_Gorenstein_projective_modules.tex
\UseRawInputEncoding
\documentclass[11pt]{amsart}
\usepackage{amssymb}
\usepackage{mathrsfs}
\usepackage{amsfonts}
\usepackage{graphicx}
\usepackage{amscd}
\usepackage{amsmath,amsfonts,amssymb,amsthm,mathrsfs,xypic}
\usepackage[a4paper,margin=3cm]{geometry}
\usepackage{bm}
\usepackage{color}
\usepackage{tikz}

\usepackage{hyperref}
\usepackage{color}
\usepackage{cite}

\hypersetup{colorlinks=true,
     breaklinks=true,
     linkcolor=blue,    
     bookmarks=true,
      pageanchor=true}

\makeatletter
\renewcommand*\env@matrix[1][c]{\hskip -\arraycolsep
  \let\@ifnextchar\new@ifnextchar
  \array{*\c@MaxMatrixCols #1}}
\makeatother

\input{rgb}
\usepackage[colorinlistoftodos]{todonotes}

\newtheorem{thm}{Theorem}[section]

\newtheorem{cor}[thm]{Corollary}
\newtheorem{lem}[thm]{Lemma}
\newtheorem{exm}[thm]{Example}

\newtheorem{prop}[thm]{Proposition}
\newtheorem{rem}[thm]{Remark}
\newtheorem{defn}[thm]{Definition}

\numberwithin{equation}{section}

\newcommand{\lxr}{\longrightarrow}

\def\drawell#1#2#3#4{
\draw[draw=black,fill=gray,fill opacity=#4,line width=1pt]
(#1,0) circle [x radius=#2, y radius=#3];}

\begin{document}


\title[]{Gorenstein silting modules and Gorenstein projective modules}

\author[]{Nan Gao$^{*}$, Jing Ma, Chi-Heng Zhang\\
\ \ \ \\
\textit{D\MakeLowercase{edicated to} P\MakeLowercase{u} Z\MakeLowercase{hang on the occasion of his 60th birthday}}
}
\address{Department of Mathematics, Shanghai University, Shanghai 200444, PR China}
\thanks{2020 Mathematics Subject Classification. 18G80, 16S90, 16G10, 16E10.}
\thanks{The first author is the corresponding author.}
\email{nangao@shu.edu.cn, majingmmmm@shu.edu.cn, zchmath@shu.edu.cn}
\thanks{Supported by the National Natural Science Foundation of China (Grant No. 11771272 and 11871326).}

\date{\today}

\keywords{Gorenstein silting modules, Silting modules, Relative rigid modules, 2-term Gorenstein silting complexes, 2-term silting complexes, Gorenstein dimension}

\subjclass[2010]{}

\begin{abstract}\ (Partial) Gorenstein silting modules are introduced and investigated. It is shown that for finite dimensional algebras of finite CM-type, partial Gorenstein silting modules are in bijection with $\tau_{G}$-rigid modules; Gorenstein silting modules are the module-theoretic counterpart of 2-term Gorenstein silting complexes; and the relation between 2-term Gorenstein silting complexes, t-structures and torsion pairs in module categories. Furthermore, the corresponding version of the classical Brenner-Butler theorem in this setting are characterised; and the upper bound of the global dimension of endomorphism algebras of 2-term Gorenstein silting complexes over an algebra $A$ are also characterised by terms of the Gorenstein global dimension of $A$.
\end{abstract}

\maketitle

\vskip 10pt

\section{\bf Introduction and Preliminaries}

\vskip 5pt

Silting complexes were first introduced by Keller and Vossieck \cite{KV} to study $t\mbox{-}$structures in the bounded derived category of representations of Dynkin quivers. This kind of complexes is a generalization of tilting complexes and tilting modules. Later, 2-term silting complexes arouse the interest of many mathematicians. For examples, the connection with torsion pairs and support $\tau\mbox{-}$tilting modules
are given(\cite{HKM, MM, AIR}); the counterpart of the classical tilting theorem and the global dimension of its endomorphism algebras are characterised(\cite{BZ1, BZ2}).

\vskip 10pt

Silting modules introduced by Angeleri H\"{u}gel-Marks-Vit\'{o}ria \cite{AMV1} over an arbitrary ring who are intended to generalize tilting modules in a similar fashion as 2-term silting complexes generalize 2-term tilting complexes and also, coincide with support $\tau\mbox{-}$tilting modules for finite dimensional algebras. Adachi-Iyama-Reiten \cite{AIR} showed how 2-term silting complexes relate with silting modules, t-structures, and co-t-structures. Silting theory has been developed systematically, see for examples(\cite{MS, AMV2, A, LZ}).

\vskip 10pt

The main idea of Gorenstein homological algebra is to replace projective modules by Gorenstein-projective modules. These modules were introduced by Enochs and Jenda \cite{EJ1} as a generalization of finitely generated modules of G-dimension zero over a two-sided Noetherian ring, in the sense of Auslander and Bridger \cite{AB}. The subject has been developed to an advanced level, see for example \cite{ABu,AR,Hap,EJ2,Ch,AM,Hol,B1,CFH,BR,J, Chen, GZ, GK, RZ1, RZ2, RZ3, AS1, G1, YLO, Z, CW}.

\vskip 10pt

Beligiannis \cite{B2} introduced and studied the algebras of finite Cohen-Macaulay type (resp. finite CM-type for simply), which correspond to algebras of finite representation type in Gorenstein homological algebra. For this class of algebras, Gao \cite{G2} introduced the relative transpose ${\rm Tr}_{G}$ in terms of Gorenstein-projective modules and the corresponding Auslander-Reiten formula.

\vskip 10pt

Based on these work, we draw a picture of all relevant modules as follows:

\vskip 10pt

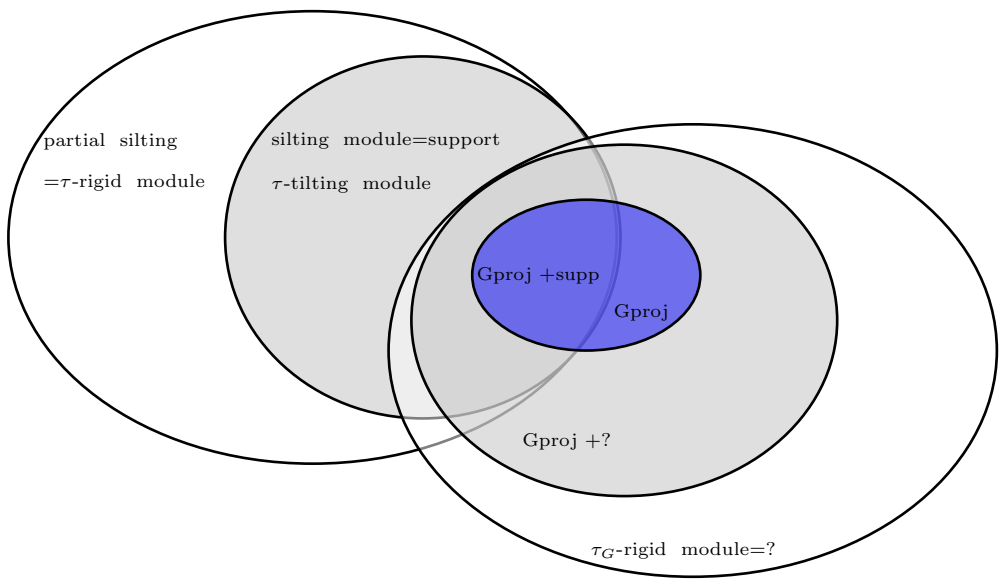
\begin{tikzpicture}[font=\rmfamily,scale=1]
\drawell{0}{4}{3}{0};
\node at (-2,1){\parbox[c]{8em}{{\tiny partial \ silting }\\
{\tiny=$\tau\mbox{-}$rigid \ module}}};

\draw[draw=black,line width=1pt,fill=gray,fill opacity=.5]
(1.45,0) circle [x radius=2.6, y radius=2.4];
\node at (1,1) {\parbox[c]{8em}{{\tiny silting \ module=support}
\vskip 1pt
{\tiny $\tau\mbox{-}$tilting \ module}}};

\draw[draw=black,line width=1pt,fill=white,fill opacity=.5]
(5,-1.5) circle [x radius=4, y radius=3];
\node at (5.2,-4.15) {\parbox[c]{8em}{{\tiny $\tau_{G}\mbox{-}$rigid \ module=?}}};

\draw[draw=black,line width=1pt,fill=gray,fill opacity=.5]
(4.1,-1.1) circle [x radius=2.8, y radius=2.33];
\node at (4.3,-2.7) {\parbox[c]{8em}{{\tiny Gproj +?}}};

\draw[draw=black,line width=1pt,fill=blue,fill opacity=.5]
(3.6,-0.5) circle [x radius=1.5, y radius=1];
\node at (3.7,-0.5) {\parbox[c]{8em}{{\tiny Gproj +supp}}};
\node at (5.5,-1) {\parbox[c]{8em}{{\tiny Gproj}}};
\end{tikzpicture}

\vskip 10pt

More precisely, there are the following natural questions:

\vskip 5pt

{\bf Question A:}\  Which class of modules coincides with the above $\tau_{G}$-rigid module?

\vskip 10pt

{\bf Question B:}\ What's the correspondence in the bounded homotopy category with respect to the above modules?

\vskip 10pt

The answers are given in the paper. We organize the paper as follows. In Section 2, we introduce the Gorenstein silting module and show that it has a bijective relationship with the relative rigid module over an algebra of finite CM-type. We also show the relations among  Gorenstein tilting modules, Gorenstein star modules  and Gorenstein silting modules. In Section 3, we  show that the Gorenstein silting module is the module-theoretic counterpart of the 2-term Gorenstein silting complex. We also characterise it by the connection with the t-structure and torsion pair, and show the corresponding Brenner-Butler theorem, and characterise the global dimension of endomorphism algebras of 2-term Gorenstein silting complexes over an algebra $A$ by terms of the Gorenstein dimension of $A$.

\vskip 10pt

First we  give the main definitions in the paper.

\vskip 10pt

For the algebra $A$, following~\cite{EJ2}, an exact sequence $G_{1}\xrightarrow{d^{1}} G_{0}\lxr M\lxr 0$ is called a proper Gorenstein-projective presentation of $M$ if each $G_{i}$ is Gorenstein-projective and ${\rm Hom}_{A}(G,$ $G_{1})\lxr {\rm Hom}_{A}(G,G_{0}) \lxr {\rm Hom}_{A}(G,T)\lxr 0$ is exact for any Gorenstein-projective module $G$.

\vskip 10pt

For a morphism $\theta: G_{1}\lxr G_{0}$ with $G_{1}$ and $G_{0}$ being Gorenstein-projective modules. We consider the class of $R\mbox{-}$modules
$$D_{\theta}:=\{X\in {\rm Mod}R\mid {\rm Hom}_{R}(\theta, X) \ {\rm is \ epic}\}.$$

{\bf Definition~\ref{Gsmodule}}\ Let $R$ be a Noetherian ring. We say that an $R\mbox{-}$module $T$ is

\vskip 5pt

$\bullet$ \ partial Gorenstein silting if there is a proper Gorenstein-projective presentation $\theta$ of $T$ such that

\vskip 5pt

\ \ (Gs1) $D_{\theta}$ is a relative torsion class (i.e. closed for $G\mbox{-}$epimorphic images, $G\mbox{-}$extensions and coproducts);

\vskip 5pt

\ \ (Gs2) $T$ lies in $D_{\theta}$.

\vskip 5pt

$\bullet$ \ Gorenstein silting if there is a proper Gorenstein-projective presentation $\theta$ of $T$ such that ${\rm Gen}_{G}(T)=D_{\theta}$.

\vskip 10pt

Let $M$ be an $A\mbox{-}$module. Then there is a minimal Gorenstein-projective presentation
$G_{1}\lxr G_{0}\lxr M\lxr 0$ of $M$. This induces  the following exact sequences
$$0\lxr {\rm Hom}_{A}(M, E)\lxr {\rm Hom}_{A}(G_{0}, E)\lxr {\rm Hom}_{A}(G_{1}, E)\lxr {\rm Tr}_{G}M\lxr 0$$
and
$$0\lxr \tau_{G}M\lxr D{\rm Hom}_{A}(G_{1}, E)\lxr D{\rm Hom}_{A}(G_{0}, E),$$
with ${\rm Tr}_{G}M\in {\rm A(Gproj)}\mbox{-}{\rm mod}$ and $\tau_{G}M\in {\rm mod}{\rm A(Gproj)}$.  We denote by ${\rm A(Gproj)}$ the Cohen-Macaulay Auslander algebra of $A$.

\vskip 10pt

{\bf Definition~\ref{rigid}}\ Let $A$ be a finite dimensional $k\mbox{-}$algebra of finite CM-type with the Gorenstein-projective generator $E$. We say that an $A\mbox{-}$module $M$ is $\tau_{G}\mbox{-}$rigid if ${\rm Hom}_{{\rm A(Gproj)}}$ $((E,M), \tau_{G}M)=0$ for the left ${\rm A(Gproj)}\mbox{-}$module ${\rm Hom}_{A}(E,M)$.

\vskip 10pt

{\bf Definition~\ref{Gscomplex}}\ Let  $G^{\bullet}: G_{1}\stackrel{d^{1}}{\lxr} G_{0}$ be a complex with $G_{i}\in {\rm Gproj}A$ for $i=0,1$. We say that $G^{\bullet}$ is

\vskip 5pt

$\bullet$ \ 2-term partial Gorenstein silting in $K^{b}({\rm Gproj}A)$ if it satisfies the following two conditions:

\vskip 5pt

{\rm (i)} $G_{1}\lxr {\rm Im}d^{1}$ and $G_{0}\lxr {\rm Coker}d^{1}$ are right ${\rm Gproj}A$-approximations;

\vskip 5pt

{\rm (ii)} ${\rm Hom}_{D_{gp}^{b}(A)}(G^{\bullet}, G^{\bullet}[1])=0$.

\vskip 5pt

$\bullet$ \ 2-term Gorenstein silting in $K^{b}({\rm Gproj}A)$ if it is a  2-term partial Gorenstein silting complex and ${\rm thick}G^{\bullet}=D_{gp}^{b}(A)$.

\vskip 10pt

Our main theorems are as follows:

\vskip 5pt

{\bf Theorem A}\ (Theorem~\ref{properties} and~\ref{equivalence}) \ Let $A$ be a finite dimensional algebra of finite CM-type with the Gorenstein-projective generator $E$. Let $M$ be an $A\mbox{-}$module in ${\rm mod}A$.  Then the following statements are equivalent.
\begin{enumerate}

\item $M$ is $\tau_{G}\mbox{-}$rigid.

\vskip 5pt

\item ${\rm Hom}_{A}(E, M)$ is a $\tau\mbox{-}$rigid module, where $\tau$ is the Auslander-Reiten translation over ${\rm A(Gproj)}$.

\vskip 5pt

\item $M$ is  a partial Gorenstein silting $A$-module.
\end{enumerate}

\vskip 5pt

Moreover, $\tau_{G}M\cong \tau{\rm Hom}_{A}(E, M).$
\vskip 10pt

Let $G^{\bullet}: G_{1}\stackrel{d^{1}}{\lxr} G_{0}$ be a complex in $D_{gp}^{b}(A)$ with $G_{i}\in {\rm Gproj}A$ for $i=0,1$. Consider the subcategories of ${\rm mod}A:$
$$\mathcal{T}(G^{\bullet})=\{X\in {\rm mod}A\mid {\rm Hom}_{D_{gp}^{b}(A)}(G^{\bullet}, X[1])=0\}$$
and
$$\mathcal{F}(G^{\bullet})=\{Y\in {\rm mod}A\mid {\rm Hom}_{D_{gp}^{b}(A)}(G^{\bullet}, Y)=0\}.$$
Let $B={\rm End}_{D_{gp}^{b}(A)}(G^{\bullet})^{op}$. Consider the subcategories of ${\rm mod}B$
$$\mathcal{X}(G^{\bullet})= {\rm Hom}_{D_{gp}^{b}(A)}(G^{\bullet},\ \mathcal{F}(G^{\bullet})[1]) \ \ \ and \ \ \ \mathcal{Y}(G^{\bullet})= {\rm Hom}_{D_{gp}^{b}(A)}(G^{\bullet},\ \mathcal{T}(G^{\bullet})).$$

\vskip 5pt

{\bf Theorem B}\ (Theorem~\ref{torsionpair}, Theorem~\ref{torsioninb} and Theorem~\ref{gldim})\ Let $G^{\bullet}: G_{1}\lxr G_{0}$ be a 2-term Gorenstein silting complex in $D_{gp}^{b}(A)$, where $G_{i}\in {\rm Gproj}A$. Then the following statements hold.
\begin{enumerate}

\item $(\mathcal{T}(G^{\bullet}), \mathcal{F}(G^{\bullet}))$ is a torsion pair for ${\rm mod}A$.

\vskip 5pt

\item $(\mathcal{X}(G^{\bullet}), \mathcal{Y}(G^{\bullet}))$ is a torsion pair in ${\rm mod}B$ and there are equivalences
$${\rm Hom}_{D_{gp}^{b}(A)}(G^{\bullet}, -): \mathcal{T}(G^{\bullet})\lxr \mathcal{Y}(G^{\bullet}),$$
and
$${\rm Hom}_{D_{gp}^{b}(A)}(G^{\bullet}, -[1]): \mathcal{F}(G^{\bullet})\lxr \mathcal{X}(G^{\bullet}).$$

\vskip 5pt

\item ${\rm gldim} B\leq {\rm Gdim}A+1$.
\end{enumerate}

\vskip 10pt

Throughout $A$ is a finite dimensional $k\mbox{-}$algebra over a field $k$, and ${\rm mod}A$ is the category of finitely generated left $A\mbox{-}$modules. A module $G$ of ${\rm mod}A$ is Gorenstein-projective if there is an exact sequence
$$\cdots \lxr P^{-1}\lxr P^{0}\stackrel{d^{0}}{\lxr} P^{1}\lxr \cdots$$
of projective modules of ${\rm mod}A$, which stays exact after applying ${\rm Hom}_{A}(-, P)$ for each projective module $P$, such that $G\cong {\rm Ker}d^{0}$ (see \cite{EJ2}). Denote by ${\rm Gproj}A$ and ${\rm proj}A$ the full subcategories of ${\rm mod}A$ consisting of Gorenstein-projective modules and projective modules, respectively.

\vskip 10pt

Let $\mathcal{A}$ be an abelian category and $\mathcal{X},\ \mathcal{Y}$ the full additive subcategories of $\mathcal{A}$. Let $M$ be an object of $\mathcal{A}$. Following~\cite{AS2}, a morphism $f:X\lxr M$ with $X\in \mathcal{X}$ is called a right $\mathcal{X}\mbox{-}$approximation of $M$ if any morphism from an object $\mathcal{X}$ to $M$ factors through $f$. $\mathcal{X}$ is called contravariantly finite if any object in $\mathcal{A}$ admits a right $\mathcal{X}\mbox{-}$approximation. A morphism $g:M\lxr Y$ with $Y\in \mathcal{Y}$ is called a left $\mathcal{Y}\mbox{-}$approximation of $M$ if any morphism from $M$ to an object $\mathcal{Y}$ factors through $g$. $\mathcal{Y}$ is called covariantly finite if any object in $\mathcal{A}$ admits a left $\mathcal{Y}\mbox{-}$approximation.

\vskip 10pt

Following~\cite{EJ2}, an exact sequence $G_{1}\xrightarrow{d^{1}} G_{0}\lxr M\lxr 0 \ (*)$ is called a proper Gorenstein-projective presentation of $M$ if each $G_{i}$ is Gorenstein-projective and ${\rm Hom}_{A}(G,$ $G_{1})\lxr {\rm Hom}_{A}(G,G_{0}) \lxr {\rm Hom}_{A}(G,T)\lxr 0$ is exact for any Gorenstein-projective module $G$. $(*)$ is minimal if $G_{1}\lxr {\rm Im}d^{1}$ and $G_{0}\lxr M$ are right ${\rm Gproj}A\mbox{-}$approximation.  Moreover, the exact sequence $0\lxr G_{n}\lxr \ldots \lxr G_{1} \lxr G_{0} \rightarrow M \rightarrow 0$  is called a proper Gorenstein-projective resolution of $M$ of length $n$ for some non-negative integer $n$, if each $G_{i}$ is all Gorenstein-projective and $0\lxr {\rm Hom}_{A}(G, G_{n})\lxr \cdots \lxr {\rm Hom}_{A}(G,G_{0}) \lxr {\rm Hom}_{A}(G,M)\lxr 0$ is exact for any Gorenstein-projective module $G$. We say that $M$ has Gorenstein-projective dimension $d$, denoted by ${\rm Gpd}M$, if $d$ is the least, and  the Gorenstein dimension of $A$ is defined as follows:
$${\rm Gdim}A={\rm sup}\{{\rm Gpd}M|M\in {\rm mod}A\}.$$

\vskip 10pt

Now we write $C^{b}(A),\ K^{b}(A)$ and $D^{b}(A)$ for the bounded complex category, bounded homotopy category and bounded derived category of ${\rm mod}A$, respectively. Denote by $K^{b}({\rm Gproj}A)$ (resp. $K^{b}({\rm proj}A)$) the corresponding bounded homotopy category of Gorenstein-projective modules (resp. projective modules).

\vskip 10pt

\subsection{ Gorenstein derived category}
\vskip 5pt

A complex $C^{\bullet}\in C^{b}(A)$  is $\mathcal{GP}\mbox{-}$acyclic, if ${\rm Hom}_{A}(G,$ $ C^{\bullet})$ is acyclic for each $G\in {\rm Gproj}A$.  A chain map $f^{\bullet}: X^{\bullet}\lxr Y^{\bullet}$ is a $\mathcal{GP}\mbox{-}$quasi-isomorphism, if ${\rm Hom}_{A}(G, f^{\bullet})$ is a quasi-isomorphism for each $G\in {\rm Gproj}A$, i.e., there are isomorphisms of abelian groups for any $n\in \mathbb{Z}$,
$$H^{n}{\rm Hom}_{A}(G, f^{\bullet}): H^{n}{\rm Hom}_{A}(G, X^{\bullet})\cong H^{n}{\rm Hom}_{A}(G, Y^{\bullet}).$$
Put
$$K^{b}_{gpac}(A):=\{ X^{\bullet}\in K^{b}(A)\mid X^{\bullet} \ is \ \mathcal{GP}\mbox{-}\ acyclic\}.$$
Then $K^{b}_{gpac}(A)$ is a thick triangulated subcategory of $K^{b}(A)$. Following~\cite{GZ}, we have the following definition:
$$D_{gp}^{b}(A):=K^{b}(A)/ K^{b}_{gpac}(A),$$
which is called the bounded Gorenstein derived category.

\vskip 10pt

Following~\cite{AS1}, an exact sequence $0\lxr X^{\bullet}\stackrel{f}{\lxr} Y^{\bullet}\stackrel{g}{\lxr} Z^{\bullet}\lxr 0$ in  $C^{b}(A)$ is $G\mbox{-}$exact if and only if $0\lxr {\rm Hom}_{A}(G, X^{\bullet})\lxr {\rm Hom}_{A}(G, Y^{\bullet})\lxr {\rm Hom}_{A}(G, Z^{\bullet})\lxr 0$ is exact for all $G\in {\rm Gproj}A$. In the module case, $g$ is called a $G$-epimorphism and $X$ is called a $G$-submodule of $Y$. From \cite{GZ}, if the exact sequence $0\lxr X\lxr Y\lxr Z\lxr 0$ in ${\rm mod}A$ is $G\mbox{-}$exact, then $X\lxr Y\lxr Z\lxr X[1]$ is a triangle in $D_{gp}^{b}(A)$.

\vskip 10pt

\subsection{\bf CM-Auslander algebra}
\vskip 5pt

Recall from \cite{B2} that  $A$ is of finite Cohen-Macaulay type (resp. finite CM-type for simply), if there are only finitely many isomorphism classes of indecomposable finitely generated Gorenstein-projective $A$-modules. Let $\{E_{i} \}_{i=1}^{n}$ be all non-isomorphic finitely generated Gorenstein-projective $A\mbox{-}$modules and $E=\mathop{\oplus}\limits_{i=1}^{n}E_{i}$, and ${\rm A(Gproj)}= {\rm End}_{A}(E)^{op}$. Then $E$ is an $A\mbox{-}{\rm A(Gproj)}\mbox{-}$bimodule, and ${\rm A(Gproj)}$ is called the the Cohen-Macaulay Auslander (resp. CM-Auslander for simply) algebra of $A$.

\vskip 10pt

\subsection{\bf torsion pairs and t-structures}

\vskip 5pt

\begin{defn}\ {\rm (\cite{D})}\  A pair $(\mathcal{T}, \mathcal{F})$ of full subcategories of ${\rm mod}A$ is called a torsion pair provided that:

\vskip 5pt

\begin{enumerate}
\item $\mathcal{T}\cap\mathcal{F}={0}$;

\vskip 5pt

\item If $T\lxr Y\lxr 0$ is exact with $T\in \mathcal{T}$ then $Y\in \mathcal{T}$;

\vskip 5pt

\item If $0\lxr Y\lxr F$ is exact with $F\in \mathcal{F}$ then $Y\in \mathcal{F}$;

\vskip 5pt

\item For each $A\mbox{-}$module $X$ there is an exact sequence
$$0\lxr T\lxr X\lxr F\lxr 0$$
with $T\in \mathcal{T}$ and $F\in \mathcal{F}$.
\end{enumerate}
\end{defn}

\vskip 10pt

\begin{defn}\ {\rm (\cite{BBD})}\  A pair $(\mathcal{X}, \mathcal{Y})$ of subcategories of the triangulated category $D_{gp}^{b}(A)$ is called a t-structure provided that:

\vskip 5pt

\begin{enumerate}
\item $\mathcal{X}[1]\subset \mathcal{X}$ and $\mathcal{Y}[-1]\subset \mathcal{Y}$;

\vskip 5pt

\item ${\rm Hom}_{D_{gp}^{b}(A)}(X, Y[-1])=0$ for any $X\in \mathcal{X}$ and $Y\in \mathcal{Y}$;

\vskip 5pt

\item for any $C\in D_{gp}^{b}(A)$, there exists a distinguished triangle
$$X\lxr C\lxr Y[-1]\lxr X[1]$$
with $X\in \mathcal{X}$ and $Y\in \mathcal{Y}$.
\end{enumerate}
\end{defn}

\vskip 10pt

\section{\bf Gorenstein silting modules}

\vskip 5pt

In this section, we introduce and study a class of modules, which we call the Gorenstein silting module, and show that it coincides with the relative rigid module over an algebra of finite CM-type introduced by Gao (\cite{G2}). We also show the relations among  Gorenstein tilting modules, Gorenstein star modules  and Gorenstein silting modules.

\vskip 10pt

\subsection{\bf Gorenstein silting modules} \ In this subsection, we introduce the Gorenstein silting module. Before this, we make some preparation.

\vskip 5pt

Throughout this subsection, let $R$ be a Noetherian ring, and ${\rm Mod}R$ the category of all left $R\mbox{-}$modules. We denote by $Gp(R)$(resp.\ $Gi(R)$) the full subcategory of ${\rm Mod}R$ consisting of Gorenstein-projective (resp. Gorenstein-injective)  modules.

\vskip 5pt

Now we assume that $Gp(R)$ is contravariantly finite in ${\rm Mod}R$. For $R\mbox{-}$modules $M$ and $N$, we compute right derived functors of ${\rm Hom}_{R}(M, N)$ using a Gorenstein-projective resolution of $M$ (\cite{EJ2}, \cite{Hol}). We will denote these derived functors by ${\rm Gext}_{R}^{i}(M, N)$. A short exact sequence $0\lxr M\lxr N\lxr L\lxr 0$ is  $G\mbox{-}$exact if and only if it is in ${\rm Gext}_{R}^{1}(L, M)$.

\vskip 10pt

Let $T$ be an $R$-module. Put
$${\rm Pres}_{G}(T):=\{M\in {\rm Mod}R|\exists \ a \ G\mbox{-}exact \ sequence \ T_{1}\lxr T_{0}\lxr M\lxr 0$$ $$ \ with \ T_{i}\in {\rm Add}T\},$$
$${\rm Gen}_{G}(T):=\{M\in {\rm Mod}R|\exists \ a \ G\mbox{-}exact \ sequence \ T_{0}\lxr M\lxr 0 \ with \ T_{0}\in {\rm Add}T\},$$
and
$$T^{G\perp}:=\{M\in {\rm Mod}R|{\rm Gext}^{1}_{R}(T, M)=0\},\ \ \ T^{\perp 0}:=\{M\in {\rm Mod}R|{\rm Hom}_{R}(T, M)=0\},$$
where ${\rm Add}T$ denotes the subcategory of modules consisting of direct summands of direct sums of $T$

\vskip 10pt

For a morphism $\theta: G_{1}\lxr G_{0}$ with $G_{1}$ and $G_{0}$ being Gorenstein-projective modules. We consider the class of $R\mbox{-}$modules
$$D_{\theta}:=\{X\in {\rm Mod}R\mid {\rm Hom}_{R}(\theta, X) \ {\rm is \ epic}\}.$$
We first collect some useful propoerties of $D_{\theta}$.

\vskip 10pt

\begin{lem}\
\label{silting}
Let $\theta: G_{1}\lxr G_{0}$ with $G_{1}$ and $G_{0}$ being Gorenstein-projective modules, and $T$ the cokernel of $\theta$ such that $\theta$ is the  proper Gorenstein-projective presentation of $T$.

\vskip 5pt

{\rm (i)} $D_{\theta}$ is closed under $G\mbox{-}$epimorphic images, $G\mbox{-}$extensions and direct products.

\vskip 5pt

{\rm (ii)} The class $D_{\theta}$ is contained in $T^{G\perp}$.
\begin{proof}\ The statement (i) is easy to prove. Now we prove (ii). There exists the following diagram
\[
\xymatrix@C=0.5cm{
 && G_{1}\ar@{->}[rr]^{\theta}    && G_{0} \ar[rr]^{} \ar@{<-_{)}}[dl]_{i} &&T \ar[rr]      &&0\\
   &&& {\rm Im}\theta \ar@{<<-}[ul]^{\pi}  &&& \\
}
\]
Consider the $G\mbox{-}$exact sequence $0\lxr {\rm Im}\theta\stackrel{i}{\lxr} G_{0}\lxr T\lxr 0$. Applying the functor ${\rm Hom}_{R}(-, X)$ for any $X\in D_{\theta}$, then we get the $G\mbox{-}$exact sequence
$${\rm Hom}_{R}(G_{0}, X)\stackrel{i^{*}}{\lxr} {\rm Hom}_{R}({\rm Im}\theta, X)\lxr {\rm Gext}_{R}^{1}(T, X)\lxr 0.$$
We show that $i^{*}$ is surjective. Let $f\in {\rm Hom}_{R}({\rm Im}\theta, X)$. Since $X\in D_{\theta}$, there is a map $g: G_{0}\lxr X$ such that $f\pi= g\theta= gi\pi$. Since $\pi$ is an epimorphism, we have $f= gi$. Hence ${\rm Gext}_{R}^{1}(T, X)=0$, and so $X\in T^{G\perp}$.
\end{proof}
\end{lem}

\vskip 10pt

\begin{defn}\
\label{Gsmodule}
We say that an $R\mbox{-}$module $T$ is

\vskip 5pt

$\bullet$ \ {\bf partial Gorenstein silting} if there is a proper Gorenstein-projective presentation $\theta$ of $T$ such that

\vskip 5pt

\ \ (Gs1) $D_{\theta}$ is a relative torsion class (i.e. closed for $G\mbox{-}$epimorphic images, $G\mbox{-}$extensions and coproducts);

\vskip 5pt

\ \ (Gs2) $T$ lies in $D_{\theta}$.

\vskip 5pt

$\bullet$ \ {\bf Gorenstein silting} if there is a proper Gorenstein-projective presentation $\theta$ of $T$ such that ${\rm Gen}_{G}(T)=D_{\theta}$.
\end{defn}

\vskip 10pt

\begin{prop}\ Let $T$ be an $R$-module with the proper Gorenstein-projective presentation $\theta: G_{1}\to G_{0}$. If $T$ is a partial Gorenstein silting module with respect to $\theta$, and for each $E\in Gp(R)$, there exists a $G\mbox{-}$exact sequence $E\stackrel{\phi}{\lxr} T_{0}\lxr T_{-1}\lxr 0$ with $T_{0}$ and $T_{-1}$ in ${\rm Add}T$ such that $\phi$ is the left $D_{\theta}\mbox{-}$approximation, then $T$ is a Gorenstein silting module.
\begin{proof}\  Since $T$ is a partial Gorenstein silting module with respect to $\theta$, it is clear that ${\rm Gen}_{G}(T)\subseteq D_{\theta}$. Let $X$ be an object in $D_{\theta}$. Since $Gp(R)$ is contravariantly finite, there is an $R$-module $E\in Gp(R)$ and a $G$-epimorphism $E\lxr X$. By assumption this epimorphism factors through the left $D_{\theta}\mbox{-}$approximation $\phi: E\lxr T_{0}$ via a $G$-epimorphism $g: T_{0}\lxr X$. Thus $X$ lies in ${\rm Gen}_{G}(T)$. This means that $T$ is a Gorenstein silting module.
\end{proof}
\end{prop}

\vskip 10pt

\subsection{\bf Connection with Gorenstein tilting (resp.\ star) modules} \ In this subsection, we characterise the relations among Gorenstein silting modules, Gorenstein tilting modules and Gorenstein star modules.

\vskip 10pt

Throughout, let $R$ be a Noetherian ring such that $Gp(R)$  is the contravariantly finite subcategory of ${\rm Mod}R$.

\vskip 10pt

\begin{defn}\ {\rm (\cite{AS1, G1, YLO})} \  An $R\mbox{-}$module $T$ is called a Gorenstein tilting module if $T^{G\perp}= {\rm Pres}_{G}(T)$, or equivalently, it satisfies the following three conditions:

\vskip 5pt

(T1) ${\rm Gpd}_{R}T\leq 1$.

\vskip 5pt

(T2) ${\rm Gext}^{1}_{R}(T, T^{(I)})= 0$ for  all sets $I$.

\vskip 5pt

(T3) For any $E\in Gp(R)$, there exists a $G\mbox{-}$exact sequence $0\lxr E\lxr T_{0}\lxr T_{-1}\lxr 0$ with each $T_{i}\in {\rm Add}T$.
\end{defn}

\vskip 10pt

\begin{defn}\ {\rm (\cite{Z})}\  An $R\mbox{-}$module $T$ is called a Gorenstein star module if

\vskip 5pt

{\rm (i)} Any $G\mbox{-}$exact sequence $0\lxr X\lxr T_{0}\lxr Y\lxr 0$ with $T_{0}\in {\rm Add}T$ and $X\in {\rm Gen}_{G}(T)$ is  ${\rm Hom}_{R}(T, -)\mbox{-}$exact.

\vskip 5pt

{\rm (ii)} ${\rm Gen}_{G}(T)={\rm Pres}_{G}T$.
\end{defn}

\vskip 10pt

\begin{lem}\
\label{Gorensteinstar}
{\rm (\cite[Proposition 2.9]{Z})}\ The following statements are equivalent for an $R\mbox{-}$module $T$:

\vskip 5pt

{\rm (i)} $T$ is a Gorenstein star module and ${\rm Gen}_{G}(T)$ is closed under $G\mbox{-}$extension.

\vskip 5pt

{\rm (ii)} ${\rm Gen}_{G}(T)= {\rm Pres}_{G}T\subseteq T^{G\perp}$.
\end{lem}

\vskip 10pt

\begin{prop}\
\label{relation}
The following hold.
\begin{enumerate}

\item \ Each Gorenstein tilting $R$-module is  Gorenstein silting.

\vskip 5pt

\item \ Each Gorenstein silting $R$-module is a Gorenstein star module.
\end{enumerate}
\begin{proof}\ (1) Let $T$ be a Gorenstein tilting $R$-module. Then by definition there exists a $G$-exact sequence as follows:
$$0\lxr G_{1}\stackrel{\theta}{\lxr} G_{0}\lxr T\lxr 0.$$

Let $X\in {\rm Gen}_{G}(T)$. Then there exists an $G$-epimorphism $g: T^{(I)}\lxr X$. Since ${\rm Gpd}_{R}T\leq 1$ and  ${\rm Gext}^{1}_{R}(T, T^{(I)})= 0$, applying the functor ${\rm Hom}_{R}(T, -)$ to $g$, it follows that ${\rm Gext}^{1}_{R}(T, X)= 0$, and hence $X\in D_{\theta}$.

\vskip 5pt

Let $X\in D_{\theta}$. Then by Lemma~\ref{silting} $X\in T^{G\perp}$. It follows from the definition $T^{G\perp}={\rm Pres}_{G}(T)$ that $X\in {\rm Pres}_{G}(T)$. This implies that $X\in
{\rm Gen}_{G}(T)$. Therefore, we prove that ${\rm Gen}_{G}(T)=D_{\theta}$, and $T$ is a Gorenstein silting module with respect to $\theta$.

\vskip 10pt

(2) Let $T$ be a Gorenstein silting $R$-module with the proper Gorenstein-projective presentation $\theta: G_{1}\to G_{0}$. Let $X\in {\rm Gen}_{G}(T)$. Then there is the $G$-epimorphic universal map $u: T^{(I)}\lxr X$ for some index set $I$. We will show that $K:= {\rm Ker}u$ lies in $D_{\theta}= {\rm Gen}_{G}(T)$. Pick $f: G_{1}\lxr K$. Since $T^{(I)}$ lies in $D_{\theta}$, we get the following commutative diagram of exact rows
\[
\xymatrix{
 &  G_{1}\ar[r]^{\theta}\ar[d]_{f} &  G_{0}\ar[r]^{\pi}\ar[d]_{g} & T\ar[r]^{}\ar[d]_{h} & 0 \\
0\ar[r]^{} & K\ar[r]^{v} &  T^{(I)}\ar[r]^{u} & X\ar[r]^{} & 0. }
\]
By the universality of $u$, there is the morphism $h': T\lxr  T^{(I)}$ such that $h= uh'$. Then by $ug=h\pi$ we have $u(g-h' \pi)=0$. Hence there is a map $g': G_{0}\lxr K$ such that $g-h' \pi= vg'$, and moreover,  we get from $vf=g\theta$ that $f=g'\theta$. This implies that $K\in D_{\theta}={\rm Gen}_{G}(T)$, and so ${\rm Gen}_{G}(T)\subseteq {\rm Pres}_{G}(T)$. This means that ${\rm Pres}_{G}(T)= {\rm Gen}_{G}(T)\subseteq T^{G\perp }$. Thus  Lemma ~\ref{Gorensteinstar} shows that $T$ is a Gorenstein star module.
\end{proof}
\end{prop}

\vskip 10pt

Recall that a ring  $R$ is Gorenstein if $R$  is two-sided Noetherian and has finite injective dimension, both as left and right $R\mbox{-}$module.

\vskip 10pt

\begin{thm}\
\label{star}
Let $R$ be a 1-Gorenstein ring and $T$ an $R$-module. Then the following are equivalent.
\begin{enumerate}

\item \ $T$ is a Gorenstein tilting module.

\vskip 5pt

\item \ $T$ is a Gorenstein silting module.

\vskip 5pt

\item \ $T$ is a Gorenstein star module and $Gi(R)\subset{\rm Pres}_{G}(T)$.
\end{enumerate}
\begin{proof}\ First by \cite[Theorem 3.2]{Z} we know that (1)$\Longleftrightarrow$ (3). The theorem immediately follows from
Proposition~\ref{relation}.
\end{proof}
\end{thm}

\vskip 10pt

\begin{exm}\ Let $A=k[x]/\langle x^{2}\rangle$ with $k$ a field. Then $T_{2}(A)=\begin{pmatrix}\begin{smallmatrix} A & A \\ 0 & A \\ \end{smallmatrix}\end{pmatrix}$ has 9 finite-dimensional indecomposable modules: $$G_{1}=\begin{pmatrix}\begin{smallmatrix} S \\ 0 \\ \end{smallmatrix}\end{pmatrix},\ \ \ \ G_{2}=\begin{pmatrix}\begin{smallmatrix} A \\ 0 \\ \end{smallmatrix}\end{pmatrix},\ \ \ \ G_{3}=\begin{pmatrix}\begin{smallmatrix} S \\ S \\ \end{smallmatrix}\end{pmatrix},\ \ \ \ G_{4}=\begin{pmatrix}\begin{smallmatrix} A \\ S \\ \end{smallmatrix}\end{pmatrix},$$ $$G_{5}=\begin{pmatrix}\begin{smallmatrix} A \\ A \\ \end{smallmatrix}\end{pmatrix}_{\rm id},\ \ \ \ G_{6}=\begin{pmatrix}\begin{smallmatrix} 0 \\ A \\ \end{smallmatrix}\end{pmatrix},\ \ \ \ G_{7}=\begin{pmatrix}\begin{smallmatrix} S \\ A \\ \end{smallmatrix}\end{pmatrix},\ \ \ \ G_{8}=\begin{pmatrix}\begin{smallmatrix} 0 \\ S \\ \end{smallmatrix}\end{pmatrix},\ \ \ \ G_{9}=\begin{pmatrix}\begin{smallmatrix} A \\ A \\ \end{smallmatrix}\end{pmatrix}_{x},$$ where $G_{i},\ 1\leq i\leq 5$ are Gorenstein-projectives {\rm (see Beligiannis and Reiten \cite[p.101]{BR})}. Note that $T_{2}(A)$ is 1-Gorenstein, and $G_{8}$ is a Gorenstein silting module.
\end{exm}

\vskip 10pt

We finish this section with an important class of examples of (partial) Gorenstein silting modules: $\tau_{G}\mbox{-}$rigid modules over a finite dimensional $k$-algebra, where $\tau_{G}$ was introduced in \cite{G2} for an algebra of finite CM-type.

\vskip 10pt

\subsection{\bf Relative rigid modules over algebras of finite CM-type}\ Throughout this section,  $A$ is a finite dimensional  $k\mbox{-}$algebra of finite CM-type over a field $k$ and ${\rm mod}A$ is the category of finitely generated left $A$-modules. All $A\mbox{-}$modules we consider are in ${\rm mod}A$. Use the notation in the introduction.
Denote by ${\rm A(Gproj)}\mbox{-}{\rm mod}$ the category of finitely generated right ${\rm A(Gproj)}\mbox{-}$modules and ${\rm A(Gproj)}\mbox{-}\underline{{\rm mod}}$ the stable category of ${\rm A(Gproj)}\mbox{-}{\rm mod}$ modulo projective ${\rm A(Gproj)}\mbox{-}$modules. Let $D: {\rm A(Gproj)}\mbox{-}{\rm mod}\lxr {\rm mod}{\rm A(Gproj)}$ be the duality. For Simplicity, we denote the functors ${\rm Hom}_{A}(-,-)$ by $(-,-)$ and ${\rm Hom}_{{\rm A(Gproj)}}(-,-)$ by $_{{\rm A(Gproj)}}(-,-)$.

\vskip 10pt

Let $M$ be an $A\mbox{-}$module. Then there is a minimal Gorenstein-projective presentation
$G_{1}\lxr G_{0}\lxr M\lxr 0$ of $M$. This induces  the following exact sequences
$$0\lxr (M, E)\lxr (G_{0}, E)\lxr (G_{1}, E)\lxr {\rm Tr}_{G}M\lxr 0$$
and
$$0\lxr \tau_{G}M\lxr D(G_{1}, E)\lxr D(G_{0}, E),$$
with ${\rm Tr}_{G}M\in {\rm A(Gproj)}\mbox{-}{\rm mod}$ and $\tau_{G}M\in {\rm mod}{\rm A(Gproj)}$.
Recall from \cite{G2} that $${\rm Tr}_{G}: {\rm mod}A/Gp(A)\lxr {\rm A(Gproj)}\mbox{-}\underline{{\rm mod}}$$
is called the relative transpose  of $A$, and moreover, ${\rm Tr}_{G}$ is a faithful functor.

\vskip 10pt

\begin{defn}\
\label{rigid}
We say that an $A\mbox{-}$module $M$ is {\bf $\tau_{G}\mbox{-}$rigid} if ${\rm Hom}_{{\rm A(Gproj)}}((E,M), \tau_{G}M)=0$ for the left ${\rm A(Gproj)}\mbox{-}$module ${\rm Hom}_{A}(E,M)$.
\end{defn}

\vskip 10pt

Next we provide some properties for $\tau_{G}\mbox{-}$rigid module. For an $A\mbox{-}$module $M$, put
$$^{\perp_{0}}(\underline{\tau_{G}M}):= \{X\in {\rm mod}{\rm A(Gproj)} \mid \overline{{\rm Hom}}_{{\rm A(Gproj)}}(X, \tau_{G}M)=0 \}.$$

\vskip 10pt

\begin{prop}\
\label{properties}
Let $M$ be an $A\mbox{-}$module and $$G_{1}\stackrel{f}{\lxr} G_{0}\lxr M\lxr 0  \eqno(*)$$ the minimal proper Gorenstein-projective presentation of $M$. Then the following statements hold.
\begin{enumerate}

\item $M$ is $\tau_{G}\mbox{-}$rigid if and only if ${\rm Hom}_{A}(f, M)$ is an epimorphism.

\vskip 5pt

\item $M$ is $\tau_{G}\mbox{-}$rigid if and only if ${\rm Hom}_{A}(E, M)$ is a $\tau\mbox{-}$rigid module, where $\tau$ is the Auslander-Reiten translation of ${\rm A(Gproj)}$.

\vskip 5pt

\item $\tau_{G}M\cong \tau(E, M).$

\vskip 5pt

\item Suppose that $M$ is $\tau_{G}\mbox{-}$rigid. Then ${\rm Hom}_{A}(E, {\rm Gen}_{G}M)\subseteq {^{\perp_{0}}(\underline{\tau_{G}M})}$.
\end{enumerate}
\begin{proof}\ There is an exact sequence
$$0\lxr \tau_{G}M\lxr D(G_{1}, E)\lxr D(G_{0}, E).\eqno(**)$$
Applying ${\rm Hom}_{{\rm A(Gproj)}}(N,-)$ for any ${\rm A(Gproj)}\mbox{-}$module $N$,  we have the following commutative diagram of exact sequences:
\[
\xymatrix@C=15pt{
0\ar[r]^{} & _{{\rm A(Gproj)}}(N, \tau_{G}M)\ar[r]^{} & _{{\rm A(Gproj)}}(N, D(G_{1}, E))\ar[r]^{}\ar[ld]_{\cong} & _{{\rm A(Gproj)}}(N, D(G_{0}, E))\ar[ld]_{\cong}  &  \\
& D_{{\rm A(Gproj)}}((E, G_{1}), N)\ar[r]^{} & D_{{\rm A(Gproj)}}((E, G_{0}), N)\ar[r]^{} & D_{{\rm A(Gproj)}}((E, M), N)\ar[r]^{} & 0. }
\]
Then we get that ${\rm Hom}_{{\rm A(Gproj)}}(N, \tau_{G}M)=0$ if and only if the map
$${\rm Hom}_{{\rm A(Gproj)}}((E, G_{0}), N)\lxr {\rm Hom}_{{\rm A(Gproj)}}((E, G_{1}), N)$$
is an epimorphism.

\vskip 5pt

(1) and (2) Applying ${\rm Hom}_{A}(E, -)$ the $(*)$, we can get the projective resolution of $(E, M)$:
$$(E, G_{1})\stackrel{\sigma}{\lxr} (E, G_{0})\lxr (E, M)\lxr 0.$$
Then we have from above arguments that ${\rm Hom}_{{\rm A(Gproj)}}((E,M), \tau_{G}M)=0$ if and only if ${\rm Hom}_{{\rm A(Gproj)}}(\sigma,$ $ (E,M))$ is epic if and only if ${\rm Hom}_{A}(f, M)$ is epic. Notice that the statement that ${\rm Hom}_{{\rm A(Gproj)}}(\sigma, (E,M))$ is epic means that ${\rm Hom}_{A}(E,M)$ is a partial silting ${\rm A(Gproj)}\mbox{-}$module, equivalently, ${\rm Hom}_{A}(E,M)$ is $\tau\mbox{-}$rigid.

\vskip 5pt

(3) Consider the exact sequence
$$0\lxr \tau(E, M)\lxr  D_{{\rm A(Gproj)}}((E, G_{1}), {\rm A(Gproj)})\lxr D_{{\rm A(Gproj)}}((E, G_{0}), {\rm A(Gproj)}).$$
Then  we get from  above arguments that $\tau_{G}M\cong {\rm Hom}_{{\rm A(Gproj)}}({\rm A(Gproj)}, \tau_{G}M) \cong \tau(E, M)$.

\vskip 5pt

(4) We have proved in (2) that ${\rm Hom}_{A}(E, M)$ is a $\tau\mbox{-}$rigid module. If $Y\in {\rm Gen}_{G}M$, then there is the $G$-exact sequence $M^{(I)}\lxr Y\lxr 0$ for some index set $I$.
Therefore we have that $(E, {\rm Gen}_{G}M)\subseteq {\rm Gen}(E, M) \subseteq (E, M)^{\perp_{1}}$. Since there are the following isomorphisms
$$\begin{aligned}
{\rm Ext}_{{\rm A(Gproj)}}^{1}((E,M),\ X)&\cong D\underline{{\rm Hom}}_{{\rm A(Gproj)}}(\tau^{-1}X,\ (E,M))\\
&\cong D\overline{{\rm Hom}}_{{\rm A(Gproj)}}(X,\ \tau(E,M))\\
&\cong D\overline{{\rm Hom}}_{{\rm A(Gproj)}}(X,\ \tau_{G}M),
\end{aligned}$$
it follows that $X\in (E, M)^{\perp_{1}}$ if and only if $X\in {^{\perp_{0}}(\underline{\tau_{G}M})}$. Hence ${\rm Hom}_{A}(E, {\rm Gen}_{G}M)$ $\subseteq {^{\perp_{0}}(\underline{\tau_{G}M})}$.
\end{proof}
\end{prop}

\vskip 10pt

\begin{thm}\
\label{equivalence}
Let $M$ be an $A$-module. Then $M$ is $\tau_{G}$-rigid if and only if $M$ is  a partial Gorenstein silting $A$-module.
\begin{proof}\ Let $G_{1}\stackrel{\theta}{\lxr} G_{0}\lxr M\lxr 0$ be the minimal proper Gorenstein-projective presentation of $M$. Since
$G_{0}$ and $G_{1}$ are finitely generated, it follows that $D_{\theta}$ is closed under coproducts. Then we get from Lemma~\ref{silting}
that $D_{\theta}$ is a relative torsion class. On the other hand, we get from Proposition~\ref{properties}(1) that $M$ is $\tau_{G}$-rigid if and only if
${\rm Hom}_{A}(\theta, M)$ is surjective. The latter implies that $M\in D_{\theta}$. Thus $M$ is $\tau_{G}$-rigid if and only if $M$ is  partial Gorenstein silting.
\end{proof}
\end{thm}

\vskip 10pt

Recall from \cite{DIJ} that an  algebra $A$ is called $\tau\mbox{-}$tilting finite if it admits finite number of isomorphism classes of indecomposable $\tau\mbox{-}$rigid modules. Inspired on it, we introduce $\tau_{G}\mbox{-}$tilting finite algebras.

\vskip 10pt

\begin{defn} \ An algebra $A$  is called $\tau_{G}\mbox{-}$tilting finite if it admits finite number of isomorphism classes of indecomposable $\tau_{G}\mbox{-}$rigid modules.
\end{defn}

\vskip 10pt

\begin{cor}\ If ${\rm A(Gproj)}$ is $\tau\mbox{-}$tilting finite, then $A$ is $\tau_{G}\mbox{-}$tilting finite, and further has finite number of isomorphism classes of indecomposable partial Gorenstein silting modules.
\begin{proof} \ By Proposition~\ref{properties}(1), we know that an $A$-module $M$ is $\tau_{G}\mbox{-}$rigid if and only if ${\rm Hom}_{A}(E, M)$ is a $\tau\mbox{-}$rigid module. This implies the result by Theorem~\ref{equivalence}.
\end{proof}
\end{cor}

\vskip 10pt

\section{\bf 2-term Gorenstein silting complexes}

\vskip 5pt

In this section, we show that the 2-term Gorenstein silting complex is in bijection  with the Gorenstein silting module.
Then we characterise it by the connection with the t-structure and torsion pair. We also characterise the global dimension of endomorphism algebras of 2-term Gorenstein silting complexes over an algebra $A$ by terms of the Gorenstein dimension of $A$. Throughout, we denote by $A$ a finite dimensional $k$-algebra over a field $k$.

\vskip 10pt

\subsection{2-term Gorenstein silting complexes}\ In this subsection, we study  2-term Gorenstein silting complexes, and show the links with t-structures and torsion pairs. The Brenner-Butler theorem  is given.
\vskip 10pt

\begin{defn}\ {\rm (\cite{CW})}
\label{Gscomplex}
Let  $G^{\bullet}: G_{1}\stackrel{d^{1}}{\lxr} G_{0}$ be a complex in $D_{gp}^{b}(A)$ with $G_{i}\in {\rm Gproj}A$ for $i=0,1$. We say that $G^{\bullet}$ is

\vskip 5pt

$\bullet$ \ {\bf 2-term partial Gorenstein silting} if it satisfies the following two conditions:

\vskip 5pt

{\rm (i)} $G_{1}\lxr {\rm Im}d^{1}$ and $G_{0}\lxr {\rm Coker}d^{1}$ are right ${\rm Gproj}A$-approximations;

\vskip 5pt

{\rm (ii)} ${\rm Hom}_{D_{gp}^{b}(A)}(G^{\bullet}, G^{\bullet}[1])=0$;

\vskip 5pt

$\bullet$ \ {\bf 2-term Gorenstein silting} in $K^{b}({\rm Gproj}A)$ if it is a  2-term partial Gorenstein silting complex and ${\rm thick}G^{\bullet}= K^{b}({\rm Gproj}A)$.
\end{defn}

\vskip 10pt

Let $G^{\bullet}: G_{1}\stackrel{d^{1}}{\lxr} G_{0}$ be a complex in $D_{gp}^{b}(A)$ with $G_{i}\in {\rm Gproj}A$ for $i=0,1$. Consider the subcategories of ${\rm mod}A:$
$$\mathcal{T}(G^{\bullet})=\{X\in {\rm mod}A\mid {\rm Hom}_{D_{gp}^{b}(A)}(G^{\bullet}, X[1])=0\}$$
and
$$\mathcal{F}(G^{\bullet})=\{Y\in {\rm mod}A\mid {\rm Hom}_{D_{gp}^{b}(A)}(G^{\bullet}, Y)=0\}.$$

\vskip 10pt

Assume that $G^{\bullet}: G_{1}\lxr G_{0}$ is a 2-term Gorenstein silting complex in  $K^{b}({\rm Gproj}A)$. Consider the subcategories of $D_{gp}^{b}(A)$:
$$D^{\leq 0}_{gp}(G^{\bullet})=\{X^{\bullet}\in D_{gp}^{b}(A)\mid {\rm Hom}_{D_{gp}^{b}(A)}(G^{\bullet}, X^{\bullet}[i])=0, \ for \ i>0\}$$
and
$$D^{\geq 0}_{gp}(G^{\bullet})=\{X^{\bullet}\in D_{gp}^{b}(A)\mid {\rm Hom}_{D_{gp}^{b}(A)}(G^{\bullet}, X^{\bullet}[i])=0, \ for \ i<0\}$$
Then we have the following facts.

\vskip 10pt

\begin{lem}\ Let $G^{\bullet}$ be a 2-term Gorenstein silting complex in $K^{b}({\rm Gproj}A)$. Then
\begin{enumerate}

\item $(D^{\leq 0}_{gp}(G^{\bullet}),\ D^{\geq 0}_{gp}(G^{\bullet}))$ is a t-structure in $D_{gp}^{b}(A)$.

\vskip 5pt

\item $\mathcal{T}(G^{\bullet})=D^{\leq 0}_{gp}(G^{\bullet})\cap {\rm mod}A$ and $\mathcal{F}(G^{\bullet})=D^{\geq 1}_{gp}(G^{\bullet})\cap {\rm mod}A$.
\end{enumerate}
\end{lem}
\begin{proof}\ Conclusion (1) can be obtained from \cite[Section 3.3]{KY}. (2) can be obtained by definitions.
\end{proof}

\vskip 10pt

Next we consider the relations between 2-term Gorenstein silting complexes, t-structures and  torsion pairs in module categories. A key lemma is given below.

\vskip 10pt

\begin{lem}\
\label{exactseq}\
For any $X^{\bullet}\in D_{gp}^{b}(A)$ and $n\in \mathbb{Z}$, we have a functorial exact sequence
$$0\lxr {\rm Hom}_{D_{gp}^{b}(A)}(G^{\bullet},\ H^{n-1}(X^{\bullet})[1])\to {\rm Hom}_{D_{gp}^{b}(A)}(G^{\bullet},\ X^{\bullet}[n])$$
$$\lxr {\rm Hom}_{D_{gp}^{b}(A)}(G^{\bullet},\ H^{n}(X^{\bullet}))\lxr 0.$$
\end{lem}
\begin{proof}\ For $X^{\bullet}[n]\in D_{gp}^{b}(A)$, applying ${\rm Hom}_{D_{gp}^{b}(A)}(-, X^{\bullet}[n])$ to a distinguished triangle
$$G_{1}\stackrel{d^{1}}{\lxr} G_{0}\lxr G^{\bullet}\lxr G_{1}[1],$$
we have a short exact sequence
$$0\lxr {\rm Coker}({\rm Hom}_{D_{gp}^{b}(A)}(d^{1},\ X^{\bullet}[n-1]))\lxr {\rm Hom}_{D_{gp}^{b}(A)}(G^{\bullet},\ X^{\bullet}[n])$$
$$\lxr {\rm Ker}({\rm Hom}_{D_{gp}^{b}(A)}(d^{1},\ X^{\bullet}[n]))\lxr 0.$$
Since
$$\begin{aligned}
{\rm Ker}({\rm Hom}_{D_{gp}^{b}(A)}(d^{1},\ X^{\bullet}[n]))&\cong {\rm Ker}({\rm Hom}_{A}(d^{1},\ H^{n}(X^{\bullet})))\\
&\cong {\rm Hom}_{D_{gp}^{b}(A)}(G^{\bullet},\ H^{n}(X^{\bullet})),
\end{aligned}$$
and
$$\begin{aligned}
{\rm Coker}({\rm Hom}_{D_{gp}^{b}(A)}(d^{1},\ X^{\bullet}[n-1]))&\cong {\rm Coker}({\rm Hom}_{A}(d^{1},\ H^{n-1}(X^{\bullet})))\\
&\cong {\rm Hom}_{D_{gp}^{b}(A)}(G^{\bullet},\ H^{n-1}(X^{\bullet})[1]),
\end{aligned}$$
we get the desired exact sequence.
\end{proof}

\vskip 10pt

\begin{prop}\
\label{heart}
Let $G^{\bullet}$ be a 2-term Gorenstein silting complex in $K^{b}({\rm Gproj}A)$, and  $\mathcal{C}_{gp}(G^{\bullet}):=D_{gp}^{\leq 0}(G^{\bullet})\cap D_{gp}^{\geq 0}(G^{\bullet})$ the heart of the induced t-structure $(D_{gp}^{\leq 0}(G^{\bullet}),\ D_{gp}^{\geq 0}(G^{\bullet}))$. Let $B={\rm End}_{D_{gp}^{b}(A)}(G^{\bullet})^{op}$.
\vskip 5pt

\vskip 5pt
\begin{enumerate}
\item $\mathcal{C}_{gp}(G^{\bullet})$ is an abelian category and the short exact sequences in $\mathcal{C}_{gp}(G^{\bullet})$ are precisely the triangles in $D_{gp}^{b}(A)$ all of whose vertices are objects in $\mathcal{C}_{gp}(G^{\bullet})$.

\vskip 5pt

\item For a complex $X^{\bullet}$ in $D_{gp}^{b}(A)$, we have that $X^{\bullet}$ is in $\mathcal{C}_{gp}(G^{\bullet})$ if and only if $H^{0}(X^{\bullet})$ is in $\mathcal{T}(G^{\bullet}),\ H^{-1}(X^{\bullet})$ is in $\mathcal{F}(G^{\bullet})$ and $H^{i}(X^{\bullet})=0$ for $i\neq -1,0$.

\vskip 5pt

\item  The functor ${\rm Hom}_{D_{gp}^{b}(A)}(G^{\bullet}, -): \mathcal{C}_{gp}(G^{\bullet})\lxr {\rm mod}B$ is an equivalence of abelian categories.

\vskip 5pt

\item For any $E\in {\rm Gproj}A$, there is a triangle in $D_{gp}^{b}(A)$
$$E\stackrel{e}{\lxr} G^{\prime\bullet }\stackrel{f}{\lxr} G^{ \prime \prime\bullet}\stackrel{g}{\lxr} E[1] \ \ \ \ \ \ \ \ \ \ \ \ (\triangle_{G^{\bullet}})$$
with $G^{\prime\bullet }, G^{\prime \prime\bullet }\in {\rm add}G^{\bullet}$.

\vskip 5pt

\item Suppose that $A$ is of finite CM-type with the Gorenstein-projective generator $E$. Then $$Q^{\bullet}: {\rm Hom}_{D_{gp}^{b}(A)}(G^{\bullet}, G^{\prime \bullet})\xrightarrow{{\rm Hom}_{D_{gp}^{b}(A)}(G^{\bullet}, f)} {\rm Hom}_{D_{gp}^{b}(A)}(G^{\bullet}, G^{\prime \prime \bullet})$$ is a 2-term partial silting complex in $K^{b}({\rm proj} B)$, where $f$ is the map from the triangle $\triangle_{G^{\bullet}}$ of the Gorenstein-projective generator $E$.
\end{enumerate}
\end{prop}
\begin{proof}\ (1)\ The pair $(D_{gp}^{\leq 0}(G^{\bullet}), D_{gp}^{\geq 0}(G^{\bullet}))$ is a t-structure in $D_{gp}^{b}(A)$. Then we come to the conclusion.

\vskip 5pt

(2)\ From Lemma~\ref{exactseq}, we have that
$$\begin{aligned}
& \ \ \ \ D_{gp}^{\leq 0}(G^{\bullet})\\
&=\{X^{\bullet}\in D_{gp}^{b}(A)\mid{\rm Hom}_{D_{gp}^{b}(A)}(G^{\bullet}, H^{i}(X^{\bullet}))=0\ {\rm for}\ i>0, \ {\rm Hom}_{D_{gp}^{b}(A)}(G^{\bullet}, H^{j}(X^{\bullet})[1])=0\ {\rm for}\ j\geq0 \}\\
&=\{X^{\bullet}\in D_{gp}^{b}(A)\mid H^{i}(X^{\bullet})=0\ {\rm for}\ i>0\ {\rm and}\ {\rm Hom}_{D_{gp}^{b}(A)}(G^{\bullet},\ H^{0}(X^{\bullet})[1])=0\}\\
&=\{X^{\bullet}\in D_{gp}^{b}(A)\mid H^{i}(X^{\bullet})=0\ {\rm for}\ i>0\ {\rm and}\ H^{0}(X^{\bullet})\in \mathcal{T}(G^{\bullet})\},
\end{aligned}$$
and
$$\begin{aligned}
& \ \ \ \ D_{gp}^{\geq 0}(G^{\bullet})\\
&=\{X^{\bullet}\in D_{gp}^{b}(A)\mid {\rm Hom}_{D_{gp}^{b}(A)}(G^{\bullet}, H^{i}(X^{\bullet}))=0\ {\rm for}\ i<0, \ {\rm Hom}_{D_{gp}^{b}(A)}(G^{\bullet}, H^{j}(X^{\bullet})[1])=0\ {\rm for}\ j<-1 \}\\
&=\{X^{\bullet}\in D_{gp}^{b}(A)\mid H^{i}(X^{\bullet})=0\ {\rm for}\ i<-1\ {\rm and}\ {\rm Hom}_{D_{gp}^{b}(A)}(G^{\bullet},\ H^{-1}(X^{\bullet}))=0\}\\
&=\{X^{\bullet}\in D_{gp}^{b}(A)\mid H^{i}(X^{\bullet})=0\ {\rm for}\ i<-1\ {\rm and}\ H^{-1}(X^{\bullet})\in \mathcal{F}(G^{\bullet})\}.
\end{aligned}$$
Therefore, we get that $X^{\bullet}\in \mathcal{C}_{gp}(G^{\bullet})$ if and only if $H^{0}(X^{\bullet})\in \mathcal{T}(G^{\bullet}),\ H^{-1}(X^{\bullet})\in \mathcal{F}(G^{\bullet})$ and $H^{i}(X^{\bullet})=0$ for $i\neq -1,0$.

\vskip 5pt

(3)\ The proof can be found in \cite[Theorem 1.3]{HKM}.

\vskip 5pt

(4)\ Let $G^{\prime \prime \bullet}\lxr E[1]$ be a right add$G^{\bullet}$-approximation of $E[1]$. Extend it to a triangle
$$E\lxr H^{\bullet}\lxr G^{\prime \prime \bullet}\lxr E[1],\eqno(*)$$
where $H^{\bullet}$ is a 2-term complex in $K^{b}({\rm Gproj}A)$. By applying the functors ${\rm Hom}_{D_{gp}^{b}(A)}(G^{\bullet}, -)$ and ${\rm Hom}_{D_{gp}^{b}(A)}(-, G^{\bullet})$ to the triangle $(*)$, we have that
$${\rm Hom}_{D_{gp}^{b}(A)}(G^{\bullet}, H^{\bullet}[1])=0 \ \ \ and\ \ \ {\rm Hom}_{D_{gp}^{b}(A)}(H^{\bullet}, G^{\bullet}[1])=0.$$
Applying ${\rm Hom}_{D_{gp}^{b}(A)}(-, H^{\bullet})$ yields ${\rm Hom}_{D_{gp}^{b}(A)}(H^{\bullet}, H^{\bullet}[1])=0$. Hence
$G^{\bullet}\oplus H^{\bullet}$ is a 2-term partial Gorenstein silting complex in $K^{b}({\rm Gproj}A)$. The triangle $(*)$ shows that $E\in {\rm thick}(G^{\bullet}\oplus H^{\bullet})$ and so $G^{\bullet}\oplus H^{\bullet}$ is a 2-term Gorenstein silting complex. Therefore, we get the desired triangle $\triangle_{G^{\bullet}}$.

\vskip 5pt

(5) \ Let $\alpha$ be a morphism in ${\rm Hom}_{K^{b}({\rm proj}B)}(Q^{\bullet}, Q^{\bullet}[1])$. Then it has the following form
\[
\xymatrix@C=55pt{
 & {\rm Hom}_{D_{gp}^{b}(A)}(G^{\bullet}, G^{\prime \bullet})\ar[r]^{{\rm Hom}_{D_{gp}^{b}(A)}(G^{\bullet}, f)}\ar[d]^{\alpha} & {\rm Hom}_{D_{gp}^{b}(A)}(G^{\bullet}, G^{\prime \prime \bullet})  \\
{\rm Hom}_{D_{gp}^{b}(A)}(G^{\bullet}, G^{\prime \bullet})\ar[r]^{{\rm Hom}_{D_{gp}^{b}(A)}(G^{\bullet}, f)} & {\rm Hom}_{D_{gp}^{b}(A)}(G^{\bullet}, G^{\prime \prime \bullet}) &, }
\]
and there is a morphism $h: G^{\prime \bullet}\lxr G^{\prime \prime \bullet}$ such that $\alpha= {\rm Hom}_{D_{gp}^{b}(A)}(G^{\bullet}, h)$. Since ${\rm Hom}_{D_{gp}^{b}(A)}(E,$ $ E[1])=0$, there are unique morphisms $h_{1},\ h_{2}$ such that the following diagram
\[
\xymatrix@C=45pt{
E\ar[r]^{e}\ar[d]^{h_{1}} & G^{\prime \bullet}\ar[r]^{f}\ar[d]^{h}\ar@{-->}[ld]_{h_{4}} & G^{\prime \prime \bullet}\ar[r]^{g}\ar[d]^{h_{2}}\ar@{-->}[ld]_{h_{3}} & E[1]\ar[d]^{h_{1}[1]}  \\
G^{\prime \bullet}\ar[r]^{f} & G^{\prime \prime \bullet}\ar[r]^{g} & E[1]\ar[r]^{-e[1]} & G^{\prime \bullet}[1]. }
\]
is commutative. So there is a morphism $h_{3}:G^{\prime \prime \bullet}\lxr G^{\prime \prime \bullet}$ such that $h_{2}=gh_{3}$, and also,
$$g(h-h_{3}f)=gh-gh_{3}f=gh-h_{2}f=0.$$
Hence there is a morphism $h_{4}$ such that $h-h_{3}f=fh_{4}$. Applying ${\rm Hom}_{D_{gp}^{b}(A)}(G^{\bullet}, -)$ to $h-h_{3}f=fh_{4}$ yields
$$\alpha= {\rm Hom}_{D_{gp}^{b}(A)}(G^{\bullet}, h_{3}){\rm Hom}_{D_{gp}^{b}(A)}(G^{\bullet}, f)+ {\rm Hom}_{D_{gp}^{b}(A)}(G^{\bullet}, f){\rm Hom}_{D_{gp}^{b}(A)}(G^{\bullet}, h_{4}),$$
which implies that $\alpha$ regarded as a map in ${\rm Hom}_{K^{b}({\rm proj}B)}(Q^{\bullet}, Q^{\bullet}[1])$ is null-homotopic. Thus, $Q^{\bullet}$ is a 2-term partial silting complex in $K^{b}({\rm proj}B)$.
\end{proof}

\vskip 10pt

Let $X,\ Y\in {\rm mod}A$ and $f:X\lxr Y$. We denote by ${\rm GIm}f$ the $G\mbox{-}$image of $f$, this is, ${\rm Hom}_{A}(G, X)\lxr {\rm Hom}_{A}(G, {\rm Im}f)$ remains to be epic for any $G\in {\rm Gproj}A$. Let $X\in {\rm mod}A$. Consider the canonical sequence of $X$:
$$0\lxr tX\stackrel{i_{X}}{\lxr} X\lxr X/tX\lxr 0,$$
where $tX=\sum {\rm GIm}f$ with $f\in {\rm Hom}_{A}(H^{0}(G^{\bullet}),\ X) $ such that any $g:E\lxr X$ factors through $f$ for any $E\in {\rm Gproj}A$. We collect some properties of $\mathcal{T}(G^{\bullet})$ and $\mathcal{F}(G^{\bullet})$.

\vskip 5pt

\begin{lem}\ The following hold:

\vskip 5pt

\begin{enumerate}
\item $\mathcal{T}(G^{\bullet})$ is closed under $G$-epimorphic images.

\vskip 5pt

\item $\mathcal{F}(G^{\bullet})$ is closed under $G$-submodules.
\vskip 5pt

\item For any $X\in {\rm mod}A,\ {\rm Hom}_{A}(H^{0}(G^{\bullet}),\ i_{X})$ is an isomorphism.
\end{enumerate}
\end{lem}
\begin{proof}\ (1)\ \ Let $0\lxr X\lxr Y\lxr Z\lxr 0$ be a $G$-exact sequence in ${\rm mod}A$. Applying ${\rm Hom}_{D_{gp}^{b}(A)}(G^{\bullet}, -)$, we get the following exact sequence
$${\rm Hom}_{D_{gp}^{b}(A)}(G^{\bullet}, Y[1])\lxr {\rm Hom}_{D_{gp}^{b}(A)}(G^{\bullet}, Z[1])\lxr {\rm Hom}_{D_{gp}^{b}(A)}(G^{\bullet}, X[2]).$$
Since ${\rm Hom}_{D_{gp}^{b}(A)}(G^{\bullet}, X[2])=0$, then we have that $Y\in \mathcal{T}(G^{\bullet})$ implies $Z\in \mathcal{T}(G^{\bullet})$.

\vskip 5pt

(2)\ \ Let $0\lxr X\lxr Y\lxr Z\lxr 0$ be a $G$-exact sequence in ${\rm mod}A$. Applying ${\rm Hom}_{D_{gp}^{b}(A)}(G^{\bullet}, -)$, we get the following exact sequence
$$0\lxr {\rm Hom}_{D_{gp}^{b}(A)}(G^{\bullet}, X)\lxr {\rm Hom}_{D_{gp}^{b}(A)}(G^{\bullet}, Y)\lxr {\rm Hom}_{D_{gp}^{b}(A)}(G^{\bullet}, Z).$$
Hence $Y\in \mathcal{F}(G^{\bullet})$ implies $X\in \mathcal{F}(G^{\bullet})$.

\vskip 5pt

(3)\ \ By the definition of $tX$, we immediately obtain the desired isomorphism.
\end{proof}

\vskip 10pt

\begin{lem}\
\label{triangle}
There exists a triangle in  $D_{gp}^{b}(A)$ for a 2-term Gorenstein silting complex $G^{\bullet}$ of the form
$$H^{-1}(G^{\bullet})[1]\lxr G^{\bullet}\lxr H^{0}(G^{\bullet})\lxr H^{-1}(G^{\bullet})[2].$$
\begin{proof}\ Let $G^{\prime\bullet}:=0\lxr {\rm Im}d^{1}\lxr G_{0}\lxr 0$. We have a $G$-exact sequence
$$0\lxr {\rm Ker}d^{1}[1]\lxr G^{\bullet}\lxr G^{\prime\bullet}\lxr 0$$
in $C^{b}(A)$. Since $G^{\prime\bullet}$ is $\mathcal{GP}$-quasi-isomorphic to
$H^{0}(G^{\bullet})$, then  $G^{\prime\bullet}\cong H^{0}(G^{\bullet})$ in $D_{gp}^{b}(A)$, we get the desired triangle in $D_{gp}^{b}(A)$ of the form
$$H^{-1}(G^{\bullet})[1]\lxr G^{\bullet}\lxr H^{0}(G^{\bullet})\lxr H^{-1}(G^{\bullet})[2].$$
\end{proof}
\end{lem}

\vskip 10pt

\begin{lem}\
\label{functorialisomorphism}
For any $X\in {\rm mod}A$, we have a functorial isomorphism
$${\rm Hom}_{D_{gp}^{b}(A)}(G^{\bullet}, X)\cong {\rm Hom}_{A}(H^{0}(G^{\bullet}), X)$$
and a monomorphism
$${\rm Hom}_{D_{gp}^{b}(A)}(H^{0}(G^{\bullet}), X[1])\lxr {\rm Hom}_{A}(G^{\bullet}, X[1]).$$
\begin{proof}\ Applying ${\rm Hom}_{D_{gp}^{b}(A)}(-, X)$ to the triangle in Lemma ~\ref{triangle}
$$H^{-1}(G^{\bullet})[1]\lxr G^{\bullet}\lxr H^{0}(G^{\bullet})\lxr H^{-1}(G^{\bullet})[2]$$
and using that there is no non-zero negative extensions between modules, we get the required isomorphism and monomorphism.
\end{proof}
\end{lem}

\vskip 10pt

\begin{thm}\
\label{torsionpair}\ The following are equivalent for a complex $G^{\bullet}: G_{1}\lxr G_{0}$ with $G_{i}\in {\rm Gproj}A$.

\vskip 5pt

\begin{enumerate}
\item $G^{\bullet}$ is a 2-term Gorenstein silting complex in $D_{gp}^{b}(A)$.

\vskip 5pt

\item $\mathcal{T}(G^{\bullet})\cap \mathcal{F}(G^{\bullet})={0}$ and $H^{0}(G^{\bullet})\in \mathcal{T}(G^{\bullet})$.

\vskip 5pt

\item $\mathcal{T}(G^{\bullet})\cap \mathcal{F}(G^{\bullet})={0}$ and $t(X)\in \mathcal{T}(G^{\bullet}), \ X/tX\in \mathcal{F}(G^{\bullet})$ for all $X\in {\rm mod}A$.

\vskip 5pt

\item $(\mathcal{T}(G^{\bullet}), \mathcal{F}(G^{\bullet}))$ is a torsion pair for ${\rm mod}A$.
\end{enumerate}
\end{thm}
\begin{proof}\ ${\rm (1)\Leftrightarrow (2)}$\ ${\rm Hom}_{D_{gp}^{b}(A)}(G^{\bullet}, G^{\bullet}[i])=0$ for all $i>0$ if and only if $H^{0}(G^{\bullet})\in \mathcal{T}(G^{\bullet})$ by Lemma~\ref{exactseq}. For any $X\in \mathcal{T}(G^{\bullet})\cap \mathcal{F}(G^{\bullet}),\ {\rm Hom}_{D_{gp}^{b}(A)}(G^{\bullet}, X[n])=0$ for all $n\in \mathbb{Z}$ and hence $X=0$. Conversely, let $X^{\bullet}\in D_{gp}^{b}(A)$ with ${\rm Hom}_{D_{gp}^{b}(A)}(G^{\bullet}, X[n])=0$ for all $n\in \mathbb{Z}$. Then by Lemma~\ref{exactseq}, $H^{n}(X^{\bullet})\in \mathcal{T}(G^{\bullet})\cap \mathcal{F}(G^{\bullet})={0}$.

\vskip 5pt

${\rm (2)\Rightarrow (3)}$\ Let $X\in {\rm mod}A$. Since $H^{0}(G^{\bullet})\in \mathcal{T}(G^{\bullet})$, it follows that $tX\in \mathcal{T}(G^{\bullet})$. Next, since  there is an isomorphism by Lemma~\ref{functorialisomorphism}
$${\rm Hom}_{D_{gp}^{b}(A)}(G^{\bullet}, X/tX)\cong {\rm Hom}_{A}(H^{0}(G^{\bullet}), X/tX),$$
and ${\rm Hom}_{A}(H^{0}(G^{\bullet}), i_{X})$ is an isomorphism, it follows that ${\rm Hom}_{D_{gp}^{b}(A)}(G^{\bullet}, X/tX)=0$ and hence $ X/tX\in \mathcal{F}(G^{\bullet})$.

\vskip 5pt

${\rm (3)\Rightarrow (4)}$\ It can be obtained by the definition.

\vskip 5pt

${\rm (4)\Rightarrow (2)}$\ We just need to prove that $H^{0}(G^{\bullet})\in \mathcal{T}(G^{\bullet})$. By Lemma~\ref{functorialisomorphism}
$$0={\rm Hom}_{D_{gp}^{b}(A)}(G^{\bullet}, \mathcal{F}(G^{\bullet}))\cong {\rm Hom}_{A}(H^{0}(G^{\bullet}), \mathcal{F}(G^{\bullet})),$$
it follows from $(\mathcal{T}(G^{\bullet}), \mathcal{F}(G^{\bullet}))$ is a torsion pair that $H^{0}(G^{\bullet})\in \mathcal{T}(G^{\bullet})$.
\end{proof}

\vskip 10pt

\begin{rem}\
\label{torequi}\
Note that the torsion pair $(\mathcal{T}(G^{\bullet}), \mathcal{F}(G^{\bullet}))$ coincides with $(D_{\theta}, T^{\perp 0})$ defined in the subsection 2.1.
\end{rem}
\begin{proof}\ Let $G^{\bullet}: G_{1}\stackrel{\theta}{\lxr} G_{0}$ be a 2-term Gorenstein silting complex in $D_{gp}^{b}(A)$, and $T=H^{0}(G^{\bullet})={\rm Coker}\theta$.  On one hand, consider the distinguished triangle in $D_{gp}^{b}(A)$
$$G_{1}\stackrel{\theta}{\lxr} G_{0}\lxr G^{\bullet}\lxr G_{1}[1].$$
Applying the functor ${\rm Hom}_{D_{gp}^{b}(A)}(-, X)$ for any $A\mbox{-}$module $X$, there is the induced exact sequence
$${\rm Hom}_{D_{gp}^{b}(A)}(G_{0}, X)\lxr {\rm Hom}_{D_{gp}^{b}(A)}(G_{1}, X)\lxr {\rm Hom}_{D_{gp}^{b}(A)}(G^{\bullet}[-1], X)\lxr 0.$$
Since ${\rm Hom}_{D_{gp}^{b}(A)}(G_{i}, X)\cong {\rm Hom}_{A}(G_{i}, X)$ with $i=0,1$, we get that $X\in D_{\theta}$ if and only if ${\rm Hom}_{D_{gp}^{b}(A)}(G^{\bullet}, X[1])=0$ if and only if $X\in \mathcal{T}(G^{\bullet})$.

\vskip 5pt

On the other hand, since
$${\rm Hom}_{A}(T, X)\cong {\rm Hom}_{D_{gp}^{b}(A)}(T, X)\cong {\rm Hom}_{D_{gp}^{b}(A)}(G^{\bullet}, X),$$
we get that $X\in T^{\perp 0}$ if and only if $X\in \mathcal{F}(G^{\bullet})$.
\end{proof}

\vskip 10pt

\begin{prop}\ Let $G^{\bullet}$ be a 2-term Gorenstein silting complex in $D_{gp}^{b}(A)$ and $(\mathcal{T}(G^{\bullet}), $ $ \mathcal{F}(G^{\bullet}))$ the torsion pair induced by $G^{\bullet}$.

\vskip 5pt

\begin{enumerate}
\item For any $X\in {\rm mod}A$,  $X\in {\rm add}H^{0}(G^{\bullet})$ if and only if $X$ is ${\rm Ext}$-projective in $\mathcal{T}(G^{\bullet})$.

\vskip 5pt

\item For any $X\in \mathcal{T}(G^{\bullet})$, there is a $G$-exact sequence $0\lxr L\lxr T_{0}\lxr X\lxr 0$ with $T_{0}\in {\rm add}H^{0}(G^{\bullet})$ and $L\in \mathcal{T}(G^{\bullet})$.
\end{enumerate}
\end{prop}
\begin{proof}\ Assume that $X$ is Ext-projective in $\mathcal{T}(G^{\bullet})$. Since $\mathcal{T}(G^{\bullet})={\rm Fac}H^{0}(G^{\bullet})$, there is a G-exact sequence
$$0\lxr L\lxr T_{0}\stackrel{\alpha}{\lxr} X\lxr 0,\eqno(**)$$
where $T_{0}\stackrel{\alpha}{\lxr} X$ is a right ${\rm add}H^{0}(G^{\bullet})\mbox{-}$approximation. Since ${\rm Hom}_{A}(H^{0}(G^{\bullet}),\ \alpha)$ is an epimorphism, we have that ${\rm Hom}_{D_{gp}^{b}(A)}(G^{\bullet}, \alpha)$ is an epimorphism. Applying ${\rm Hom}_{D_{gp}^{b}(A)}(G^{\bullet}, -)$ to $(**)$, we have an exact sequence
$${\rm Hom}_{D_{gp}^{b}(A)}(G^{\bullet}, T_{0}) \xrightarrow{{\rm Hom}_{D_{gp}^{b}(A)}(G^{\bullet}, \alpha)} {\rm Hom}_{D_{gp}^{b}(A)}(G^{\bullet}, X)\lxr {\rm Hom}_{D_{gp}^{b}(A)}(G^{\bullet}, L[1])\lxr 0.$$
Then ${\rm Hom}_{D_{gp}^{b}(A)}(G^{\bullet}, L[1])=0$ which implies that $L$ is in $\mathcal{T}(G^{\bullet})$. Thus, by assumption, the sequence $(**)$ splits, and hence $X$ is in ${\rm add}H^{0}(G^{\bullet})$.

\vskip 5pt

By the monomorphism in Lemma~\ref{functorialisomorphism}, we have that ${\rm add}H^{0}(G^{\bullet})$ is Ext-projective in $\mathcal{T}(G^{\bullet})$.
\end{proof}

\vskip 10pt

\begin{thm}\ Suppose that $A$ is a Gorenstein algebra of finite CM-type with the Gorenstein projective generator $G$, and ${\rm A(Gproj)}={\rm End}_{A}(E)^{\rm op}$. Let $G^{\bullet}: G_{1}\stackrel{\theta}{\lxr} G_{0}$ be a 2-term complex in $D_{gp}^{b}(A)$, and $T=H^{0}(G^{\bullet})={\rm Coker}\theta$. Then the following statements are equivalent.

\vskip 5pt

\begin{enumerate}
\item $T$ is a Gorenstein silting module with respect to $\theta$ in ${\rm mod}A$.

\vskip 5pt

\item $G^{\bullet}: G_{1}\stackrel{\theta}{\lxr} G_{0}$ is a 2-term Gorenstein silting complex in $D_{gp}^{b}(A)$.
\end{enumerate}
\end{thm}
\begin{proof}\ First, we claim that $T$ is a partial Gorenstein silting module with respect to $\theta$ if and only if $G^{\bullet}$ is a 2-term partial Gorenstein silting complex.

\vskip 5pt

Assume that $T$ is a partial Gorenstein silting module. By definition we know that ${\rm Hom}_{A}(\theta, T)$ is an epimorphism. Then we get that ${\rm Hom}_{{\rm A(Gproj)}}((E,\theta), (E, T))$ is an epimorphism from the proof of Proposition ~\ref{properties}(1). Let $\sigma:= {\rm Hom}_{A}(E, \theta)$ and $f:(E, G_{1})\lxr (E, T)$. By the following diagram,
\[
\xymatrix@C=1.5cm{
&  {\rm Hom}_{A}(E, G_{1}) \ar[r]^{\sigma} \ar[dr]^{f} \ar@{-->}[d]^{h} \ar@{-->}[dl]^{s_{1}}  &  {\rm Hom}_{A}(E, G_{0}) \ar@{-->}[d]^{g} \ar@{-->}[dl]^{s_{0}} \\
{\rm Hom}_{A}(E, G_{1}) \ar[r]_{\sigma}  &  {\rm Hom}_{A}(E, G_{0}) \ar[r]_{\pi} & {\rm Hom}_{A}(E, T) \\
}
\]
there exists a morphism $g:{\rm Hom}_{A}(E, G_{0})\lxr {\rm Hom}_{A}(E, T)$, such that $f=g\sigma$. Since ${\rm Hom}_{A}(E, G_{1})$ is projective, there exists a morphism $$h: {\rm Hom}_{A}(E, G_{1})\lxr {\rm Hom}_{A}(E, G_{0}),$$
such that $f=\pi h$. Similarly since ${\rm Hom}_{A}(E, G_{0})$ is projective, there exists a morphism
$$s_{0}: {\rm Hom}_{A}(E, G_{0})\lxr {\rm Hom}_{A}(E, G_{0}),$$
such that $g=\pi s_{0}$. Therefore $\pi h=\pi s_{0}\sigma$, i.e., $\pi (h-s_{0}\sigma)=0$. It follows that there is
$$s_{1}: {\rm Hom}_{A}(E, G_{1})\lxr {\rm Hom}_{A}(E, G_{1}),$$
such that $h-s_{0}\sigma=\sigma s_{1}$, which shows that $h$ is null-homotopic. This implies that ${\rm Hom}_{{\rm A(Gproj)}}((E, G^{\bullet}),\ (E, G^{\bullet})[1])=0$. Therefore, we get that ${\rm Hom}_{D_{gp}^{b}(A)}(G^{\bullet}, G^{\bullet}[1])=0$. This implies that $G^{\bullet}$ is a 2-term partial Gorenstein silting complex.

\vskip 5pt

Conversely, if $G^{\bullet}$ is a 2-term partial Gorenstein silting complex, then by the diagram above, we have that $h=s_{0}\sigma+\sigma s_{1}$. Then $f=\pi h=\pi s_{0}\sigma+\pi \sigma s_{1}=(\pi s_{0})\sigma$, which means that ${\rm Hom}_{{\rm A(Gproj)}}((E,\theta),\ (E, T))$ is an epimorphism. Therefore ${\rm Hom}_{A}(\theta, T)$ is an epimorphism, and so $T\in D_{\theta}$. This implies that $T$ is a partial Gorenstein silting module with respect to $\theta$.

\vskip 5pt

(1)$\Rightarrow$(2)\ By Theorem ~\ref{torsionpair}, we prove that $\mathcal{T}(G^{\bullet})\cap \mathcal{F}(G^{\bullet})={0}$ and $H^{0}(G^{\bullet})\in \mathcal{T}(G^{\bullet})$. From Remark ~\ref{torequi}, we have $T=H^{0}(G^{\bullet})\in D_{\theta}= \mathcal{T}(G^{\bullet})$. Let $X\in \mathcal{T}(G^{\bullet})\cap \mathcal{F}(G^{\bullet})= D_{\theta}\cap T^{\perp 0}= {\rm Gen}_{G}(T)\cap T^{\perp 0}$. Then there is a $G\mbox{-}$epimorphism $T_{0}\lxr X\lxr 0$ with $T_{0}\in {\rm Add}T$, and ${\rm Hom}_{A}(T,X)=0$. Therefore we can get from the induced exact sequence $0\lxr (X, X)\lxr (T_{0}, X)$ that $X=0$. Thus $G^{\bullet}: G_{1}\stackrel{\theta}{\lxr} G_{0}$ is a 2-term Gorenstein silting complex in $D_{gp}^{b}(A)$.

\vskip 5pt

(2)$\Rightarrow$(1)\ By the above claim, we see that $T$ is a partial Gorenstein silting module, and so ${\rm Gen}_{G}(T)\subseteq D_{\theta}$. From Proposition 3.9, for any $X\in D_{\theta}=\mathcal{T}(G^{\bullet})$, there is a $G$-exact sequence $$0\lxr L\lxr T_{0}\lxr X\lxr 0$$
with $T_{0}\in {\rm add}H^{0}(G^{\bullet})={\rm add}T$ and $L\in \mathcal{T}(G^{\bullet})$. Then we get that $X\in {\rm Gen}_{G}(T)$. Therefore $T$ is a Gorenstein silting module with respect to $\theta$ in ${\rm mod}A$.
\end{proof}

\vskip 10pt

Let $B={\rm End}_{D_{gp}^{b}(A)}(G^{\bullet})^{op}$. Consider the subcategories of ${\rm mod}B$
$$\mathcal{X}(G^{\bullet})= {\rm Hom}_{D_{gp}^{b}(A)}(G^{\bullet},\ \mathcal{F}(G^{\bullet})[1]) \ \ \ and \ \ \ \mathcal{Y}(G^{\bullet})= {\rm Hom}_{D_{gp}^{b}(A)}(G^{\bullet},\ \mathcal{T}(G^{\bullet})).$$
Then we can draw the Brenner-Butler theorem in this setting.

\vskip 5pt

\begin{thm}\
\label{torsioninb}\ Let $G^{\bullet}$ be a 2-term Gorenstein silting complex in $D_{gp}^{b}(A)$. Then $(\mathcal{X}(G^{\bullet}), \mathcal{Y}(G^{\bullet}))$ is a torsion pair in ${\rm mod}B$ and there are equivalences
$${\rm Hom}_{D_{gp}^{b}(A)}(G^{\bullet}, -): \mathcal{T}(G^{\bullet})\lxr \mathcal{Y}(G^{\bullet}),$$
and
$${\rm Hom}_{D_{gp}^{b}(A)}(G^{\bullet}, -[1]): \mathcal{F}(G^{\bullet})\lxr \mathcal{X}(G^{\bullet}).$$
The equivalences send $G$-exact sequences with terms in $\mathcal{T}(G^{\bullet})$ {\rm (resp.} $\mathcal{F}(G^{\bullet})${\rm )} to short exact sequences in ${\rm mod}B$.
\end{thm}
\begin{proof}\ This follows from Proposition ~\ref{heart} (1) and (3), using that $\mathcal{T}(G^{\bullet})\cup \mathcal{F}(G^{\bullet})\subset \mathcal{C}_{gp}(G^{\bullet})$.
\end{proof}

\vskip 10pt

We finish this section with an interesting property of the 2-term Gorenstein silting complex over a finite dimensional Gorenstein algebra $A$ of finite CM-type with the Gorenstein-projective generator $E$.

\vskip 10pt

Let $G^{\bullet}: G_{1}\stackrel{d^{1}}{\lxr} G_{0}$ be the 2-term complex over ${\rm Gproj}A$, and set
$$P^{\bullet}: {\rm Hom}_{A}(E, G_{1})\lxr {\rm Hom}_{A}(E, G_{0}).$$

\vskip 5pt

\begin{prop}\ Suppose that $A$ is a Gorenstein algebra. Then  $G^{\bullet}$ is a 2-term Gorenstein silting complex in
$D_{gp}^{b}(A)$ if and only if $P^{\bullet}$ is a 2-term silting complex in $D^{b}({\rm A(Gproj)})$.
\end{prop}
\begin{proof}\ Since ${\rm Hom}_{A}(E, -)$ is a fully faithful functor, we have that
$${\rm Hom}_{D^{b}({\rm A(Gproj)})}(P^{\bullet},\ P^{\bullet}[1])\cong {\rm Hom}_{D_{gp}^{b}(A)}(G^{\bullet}, G^{\bullet}[1]).$$
On the other hand, from \cite{GZ}, there is a triangle-equivalence $D_{gp}^{b}(A)\cong D^{b}({\rm A(Gproj)})$ induced by ${\rm Hom}_{A}(E, -)$. This completes the proof.
\end{proof}

\vskip 10pt

\subsection{\bf On global dimension} \  In this subsection, we compare the global dimension between $A$ and $B$, where  $G^{\bullet}$ is a 2-term Gorenstein silting complex in $D_{gp}^{b}(A)$, and  $B={\rm End}_{D_{gp}^{b}(A)}(G^{\bullet})^{op}$.

\vskip 10pt

Recall from \cite{IY}, for full subcategories $\mathcal{X}$ and $\mathcal{Y}$ of $D^{b}_{gp}(A)$, denote $$\mathcal{X}\ast \mathcal{Y}:=\{Z\in D^{b}_{gp}(A)\mid {\rm there\ exists\ a\ triangle}\ X\lxr Z\lxr Y\lxr X[1]$$
$${\rm in}\ D^{b}_{gp}(A)\ {\rm with}\ X\in \mathcal{X}\ {\rm and}\ Y\in \mathcal{Y}\}.$$
By the octahedral axiom, we have that $(\mathcal{X}\ast \mathcal{Y})\ast \mathcal{Z}=\mathcal{X}\ast(\mathcal{Y}\ast \mathcal{Z})$. Call $\mathcal{X}$ extension closed if $\mathcal{X}\ast \mathcal{X}=\mathcal{X}$. Now fix ${\rm GP}={\rm add}G^{\bullet}$ and ${\rm GP_{c}}={\rm GP}\cap \mathcal{C}_{gp}(G^{\bullet})$.

\vskip 10pt

\begin{lem}\
\label{c1}
$\mathcal{C}_{gp}(G^{\bullet})\subset {\rm GP}\ast {\rm GP}[1]\ast \cdots \ast {\rm GP}[d+1]$ for some non-negative integer $d$.
\end{lem}
\begin{proof}\ Note that
$$\mathcal{C}_{gp}(G^{\bullet})\subset D_{gp}^{\leq 0}(G^{\bullet})\subset {\rm GP}\ast {\rm GP}[1]\ast \cdots \ast {\rm GP}[l-1]\ast {\rm GP}[l]$$
for some $l>0$. For any $X^{\bullet}\in \mathcal{C}_{gp}(G^{\bullet})$, we have $H^{i}(X^{\bullet})=0$ for $i\neq -1,0$. Taking a projective resolution $P^{\bullet}$ of  $X^{\bullet}$, then there exists some non-negative integer $d$ such that $H^{i}(P^{\bullet})=0$ for $i>0$ or $i<-d-1$. Therefore
$${\rm Hom}_{D_{gp}^{b}(A)}(X^{\bullet},\ G^{\bullet}[i])\cong {\rm Hom}_{D_{gp}^{b}(A)}(P^{\bullet},\ G^{\bullet}[i])=0,\ i\geq d+2,$$
which implies that $X^{\bullet}\in {\rm GP}\ast {\rm GP}[1]\ast \cdots \ast {\rm GP}[d+1]$.
\end{proof}

\vskip 10pt

\begin{lem}\
\label{c2}
For the complex $X^{\bullet}\in \mathcal{C}_{gp}(G^{\bullet})\cap ({\rm GP_{c}}\ast {\rm GP_{c}}[1]\ast \cdots \ast {\rm GP_{c}}[m])$ for some $m\geq 0$, we have ${\rm pd}{\rm Hom}_{D_{gp}^{b}(A)}(G^{\bullet}, X^{\bullet})_{B}\leq m$.
\end{lem}
\begin{proof}\ Let $X_{0}^{\bullet}=X^{\bullet}$. There are triangles
$$X_{i+1}^{\bullet}\lxr O_{i}^{\bullet}\stackrel{g_{i}}{\lxr} X_{i}^{\bullet}\lxr X_{i+1}^{\bullet}[1],\ \ \ 0\leq i\leq m-1$$
where $O_{i}^{\bullet}\in {\rm GP_{c}}$ and $X_{i}^{\bullet}\in {\rm GP_{c}}\ast {\rm GP_{c}}[1]\ast \cdots \ast {\rm GP_{c}}[m-i]$. Since ${\rm Hom}_{D_{gp}^{b}(A)}(G^{\bullet}, G^{\bullet}[i])=0$ for all $i>0$, we have that $g_{i}$ is a right ${\rm GP}\mbox{-}$approximation of $X_{i}^{\bullet}$. Then we get the following induced exact sequence
$${\rm Hom}_{D_{gp}^{b}(A)}(G^{\bullet}, X_{i+1}^{\bullet})\lxr {\rm Hom}_{D_{gp}^{b}(A)}(G^{\bullet}, O_{i}^{\bullet})\xrightarrow{{\rm Hom}_{D_{gp}^{b}(A)}(G^{\bullet}, g_{i})} {\rm Hom}_{D_{gp}^{b}(A)}(G^{\bullet}, X_{i}^{\bullet})\lxr 0.$$
Then we get that
$${\rm pd} {\rm Hom}_{D_{gp}^{b}(A)}(G^{\bullet}, X_{i}^{\bullet})_{B}\leq {\rm pd} {\rm Hom}_{D_{gp}^{b}(A)}(G^{\bullet}, X_{i+1}^{\bullet})_{B}+1.$$
Therefore ${\rm pd}{\rm Hom}_{D_{gp}^{b}(A)}(G^{\bullet}, X^{\bullet})_{B}\leq {\rm pd}{\rm Hom}_{D_{gp}^{b}(A)}(G^{\bullet}, X_{m}^{\bullet})_{B}+m=m$.
\end{proof}

\vskip 10pt

\begin{thm}\
\label{gldim}\ Assume that $A$ has Gorenstein dimension $d$ for some positive integer $d$. Then ${\rm gldim} B\leq d+1$.
\end{thm}
\begin{proof}\ If $G^{\bullet}$ is Gorenstein tilting, then $G^{\bullet}\in \mathcal{C}_{gp}(G^{\bullet})$. Therefore we have that ${\rm GP}={\rm GP_{c}}$. It follows from Lemma~\ref{c1} and \ref{c2} that ${\rm gldim} B\leq {\rm Gdim}A+1$.
\end{proof}

\vskip 20pt

{\bf Declarations}

\vskip 5pt

{\bf Conflict of interest} On behalf of all authors, the corresponding author states that there is no conflict of interest.

\end{document}

%% file: rgb.tex
\definecolor{AliceBlue}{rgb}{0.94,0.97,1.00}
\definecolor{AntiqueWhite1}{rgb}{1.00,0.94,0.86}
\definecolor{AntiqueWhite2}{rgb}{0.93,0.87,0.80}
\definecolor{AntiqueWhite3}{rgb}{0.80,0.75,0.69}
\definecolor{AntiqueWhite4}{rgb}{0.55,0.51,0.47}
\definecolor{AntiqueWhite}{rgb}{0.98,0.92,0.84}
\definecolor{BlanchedAlmond}{rgb}{1.00,0.92,0.80}
\definecolor{BlueViolet}{rgb}{0.54,0.17,0.89}
\definecolor{CadetBlue1}{rgb}{0.60,0.96,1.00}
\definecolor{CadetBlue2}{rgb}{0.56,0.90,0.93}
\definecolor{CadetBlue3}{rgb}{0.48,0.77,0.80}
\definecolor{CadetBlue4}{rgb}{0.33,0.53,0.55}
\definecolor{CadetBlue}{rgb}{0.37,0.62,0.63}
\definecolor{CornflowerBlue}{rgb}{0.39,0.58,0.93}
\definecolor{DarkBlue}{rgb}{0.00,0.00,0.55}
\definecolor{DarkCyan}{rgb}{0.00,0.55,0.55}
\definecolor{DarkGoldenrod1}{rgb}{1.00,0.73,0.06}
\definecolor{DarkGoldenrod2}{rgb}{0.93,0.68,0.05}
\definecolor{DarkGoldenrod3}{rgb}{0.80,0.58,0.05}
\definecolor{DarkGoldenrod4}{rgb}{0.55,0.40,0.03}
\definecolor{DarkGoldenrod}{rgb}{0.72,0.53,0.04}
\definecolor{DarkGray}{rgb}{0.66,0.66,0.66}
\definecolor{DarkGreen}{rgb}{0.00,0.39,0.00}
\definecolor{DarkGrey}{rgb}{0.66,0.66,0.66}
\definecolor{DarkKhaki}{rgb}{0.74,0.72,0.42}
\definecolor{DarkMagenta}{rgb}{0.55,0.00,0.55}
\definecolor{DarkOliveGreen1}{rgb}{0.79,1.00,0.44}
\definecolor{DarkOliveGreen2}{rgb}{0.74,0.93,0.41}
\definecolor{DarkOliveGreen3}{rgb}{0.64,0.80,0.35}
\definecolor{DarkOliveGreen4}{rgb}{0.43,0.55,0.24}
\definecolor{DarkOliveGreen}{rgb}{0.33,0.42,0.18}
\definecolor{DarkOrange1}{rgb}{1.00,0.50,0.00}
\definecolor{DarkOrange2}{rgb}{0.93,0.46,0.00}
\definecolor{DarkOrange3}{rgb}{0.80,0.40,0.00}
\definecolor{DarkOrange4}{rgb}{0.55,0.27,0.00}
\definecolor{DarkOrange}{rgb}{1.00,0.55,0.00}
\definecolor{DarkOrchid1}{rgb}{0.75,0.24,1.00}
\definecolor{DarkOrchid2}{rgb}{0.70,0.23,0.93}
\definecolor{DarkOrchid3}{rgb}{0.60,0.20,0.80}
\definecolor{DarkOrchid4}{rgb}{0.41,0.13,0.55}
\definecolor{DarkOrchid}{rgb}{0.60,0.20,0.80}
\definecolor{DarkRed}{rgb}{0.55,0.00,0.00}
\definecolor{DarkSalmon}{rgb}{0.91,0.59,0.48}
\definecolor{DarkSeaGreen1}{rgb}{0.76,1.00,0.76}
\definecolor{DarkSeaGreen2}{rgb}{0.71,0.93,0.71}
\definecolor{DarkSeaGreen3}{rgb}{0.61,0.80,0.61}
\definecolor{DarkSeaGreen4}{rgb}{0.41,0.55,0.41}
\definecolor{DarkSeaGreen}{rgb}{0.56,0.74,0.56}
\definecolor{DarkSlateBlue}{rgb}{0.28,0.24,0.55}
\definecolor{DarkSlateGray1}{rgb}{0.59,1.00,1.00}
\definecolor{DarkSlateGray2}{rgb}{0.55,0.93,0.93}
\definecolor{DarkSlateGray3}{rgb}{0.47,0.80,0.80}
\definecolor{DarkSlateGray4}{rgb}{0.32,0.55,0.55}
\definecolor{DarkSlateGray}{rgb}{0.18,0.31,0.31}
\definecolor{DarkSlateGrey}{rgb}{0.18,0.31,0.31}
\definecolor{DarkTurquoise}{rgb}{0.00,0.81,0.82}
\definecolor{DarkViolet}{rgb}{0.58,0.00,0.83}
\definecolor{DeepPink1}{rgb}{1.00,0.08,0.58}
\definecolor{DeepPink2}{rgb}{0.93,0.07,0.54}
\definecolor{DeepPink3}{rgb}{0.80,0.06,0.46}
\definecolor{DeepPink4}{rgb}{0.55,0.04,0.31}
\definecolor{DeepPink}{rgb}{1.00,0.08,0.58}
\definecolor{DeepSkyBlue1}{rgb}{0.00,0.75,1.00}
\definecolor{DeepSkyBlue2}{rgb}{0.00,0.70,0.93}
\definecolor{DeepSkyBlue3}{rgb}{0.00,0.60,0.80}
\definecolor{DeepSkyBlue4}{rgb}{0.00,0.41,0.55}
\definecolor{DeepSkyBlue}{rgb}{0.00,0.75,1.00}
\definecolor{DimGray}{rgb}{0.41,0.41,0.41}
\definecolor{DimGrey}{rgb}{0.41,0.41,0.41}
\definecolor{DodgerBlue1}{rgb}{0.12,0.56,1.00}
\definecolor{DodgerBlue2}{rgb}{0.11,0.53,0.93}
\definecolor{DodgerBlue3}{rgb}{0.09,0.45,0.80}
\definecolor{DodgerBlue4}{rgb}{0.06,0.31,0.55}
\definecolor{DodgerBlue}{rgb}{0.12,0.56,1.00}
\definecolor{FloralWhite}{rgb}{1.00,0.98,0.94}
\definecolor{ForestGreen}{rgb}{0.13,0.55,0.13}
\definecolor{GhostWhite}{rgb}{0.97,0.97,1.00}
\definecolor{GreenYellow}{rgb}{0.68,1.00,0.18}
\definecolor{HotPink1}{rgb}{1.00,0.43,0.71}
\definecolor{HotPink2}{rgb}{0.93,0.42,0.65}
\definecolor{HotPink3}{rgb}{0.80,0.38,0.56}
\definecolor{HotPink4}{rgb}{0.55,0.23,0.38}
\definecolor{HotPink}{rgb}{1.00,0.41,0.71}
\definecolor{IndianRed1}{rgb}{1.00,0.42,0.42}
\definecolor{IndianRed2}{rgb}{0.93,0.39,0.39}
\definecolor{IndianRed3}{rgb}{0.80,0.33,0.33}
\definecolor{IndianRed4}{rgb}{0.55,0.23,0.23}
\definecolor{IndianRed}{rgb}{0.80,0.36,0.36}
\definecolor{LavenderBlush1}{rgb}{1.00,0.94,0.96}
\definecolor{LavenderBlush2}{rgb}{0.93,0.88,0.90}
\definecolor{LavenderBlush3}{rgb}{0.80,0.76,0.77}
\definecolor{LavenderBlush4}{rgb}{0.55,0.51,0.53}
\definecolor{LavenderBlush}{rgb}{1.00,0.94,0.96}
\definecolor{LawnGreen}{rgb}{0.49,0.99,0.00}
\definecolor{LemonChiffon1}{rgb}{1.00,0.98,0.80}
\definecolor{LemonChiffon2}{rgb}{0.93,0.91,0.75}
\definecolor{LemonChiffon3}{rgb}{0.80,0.79,0.65}
\definecolor{LemonChiffon4}{rgb}{0.55,0.54,0.44}
\definecolor{LemonChiffon}{rgb}{1.00,0.98,0.80}
\definecolor{LightBlue1}{rgb}{0.75,0.94,1.00}
\definecolor{LightBlue2}{rgb}{0.70,0.87,0.93}
\definecolor{LightBlue3}{rgb}{0.60,0.75,0.80}
\definecolor{LightBlue4}{rgb}{0.41,0.51,0.55}
\definecolor{LightBlue}{rgb}{0.68,0.85,0.90}
\definecolor{LightCoral}{rgb}{0.94,0.50,0.50}
\definecolor{LightCyan1}{rgb}{0.88,1.00,1.00}
\definecolor{LightCyan2}{rgb}{0.82,0.93,0.93}
\definecolor{LightCyan3}{rgb}{0.71,0.80,0.80}
\definecolor{LightCyan4}{rgb}{0.48,0.55,0.55}
\definecolor{LightCyan}{rgb}{0.88,1.00,1.00}
\definecolor{LightGoldenrod1}{rgb}{1.00,0.93,0.55}
\definecolor{LightGoldenrod2}{rgb}{0.93,0.86,0.51}
\definecolor{LightGoldenrod3}{rgb}{0.80,0.75,0.44}
\definecolor{LightGoldenrod4}{rgb}{0.55,0.51,0.30}
\definecolor{LightGoldenrodYellow}{rgb}{0.98,0.98,0.82}
\definecolor{LightGoldenrod}{rgb}{0.93,0.87,0.51}
\definecolor{LightGray}{rgb}{0.83,0.83,0.83}
\definecolor{LightGreen}{rgb}{0.56,0.93,0.56}
\definecolor{LightGrey}{rgb}{0.83,0.83,0.83}
\definecolor{LightPink1}{rgb}{1.00,0.68,0.73}
\definecolor{LightPink2}{rgb}{0.93,0.64,0.68}
\definecolor{LightPink3}{rgb}{0.80,0.55,0.58}
\definecolor{LightPink4}{rgb}{0.55,0.37,0.40}
\definecolor{LightPink}{rgb}{1.00,0.71,0.76}
\definecolor{LightSalmon1}{rgb}{1.00,0.63,0.48}
\definecolor{LightSalmon2}{rgb}{0.93,0.58,0.45}
\definecolor{LightSalmon3}{rgb}{0.80,0.51,0.38}
\definecolor{LightSalmon4}{rgb}{0.55,0.34,0.26}
\definecolor{LightSalmon}{rgb}{1.00,0.63,0.48}
\definecolor{LightSeaGreen}{rgb}{0.13,0.70,0.67}
\definecolor{LightSkyBlue1}{rgb}{0.69,0.89,1.00}
\definecolor{LightSkyBlue2}{rgb}{0.64,0.83,0.93}
\definecolor{LightSkyBlue3}{rgb}{0.55,0.71,0.80}
\definecolor{LightSkyBlue4}{rgb}{0.38,0.48,0.55}
\definecolor{LightSkyBlue}{rgb}{0.53,0.81,0.98}
\definecolor{LightSlateBlue}{rgb}{0.52,0.44,1.00}
\definecolor{LightSlateGray}{rgb}{0.47,0.53,0.60}
\definecolor{LightSlateGrey}{rgb}{0.47,0.53,0.60}
\definecolor{LightSteelBlue1}{rgb}{0.79,0.88,1.00}
\definecolor{LightSteelBlue2}{rgb}{0.74,0.82,0.93}
\definecolor{LightSteelBlue3}{rgb}{0.64,0.71,0.80}
\definecolor{LightSteelBlue4}{rgb}{0.43,0.48,0.55}
\definecolor{LightSteelBlue}{rgb}{0.69,0.77,0.87}
\definecolor{LightYellow1}{rgb}{1.00,1.00,0.88}
\definecolor{LightYellow2}{rgb}{0.93,0.93,0.82}
\definecolor{LightYellow3}{rgb}{0.80,0.80,0.71}
\definecolor{LightYellow4}{rgb}{0.55,0.55,0.48}
\definecolor{LightYellow}{rgb}{1.00,1.00,0.88}
\definecolor{LimeGreen}{rgb}{0.20,0.80,0.20}
\definecolor{MediumAquamarine}{rgb}{0.40,0.80,0.67}
\definecolor{MediumBlue}{rgb}{0.00,0.00,0.80}
\definecolor{MediumOrchid1}{rgb}{0.88,0.40,1.00}
\definecolor{MediumOrchid2}{rgb}{0.82,0.37,0.93}
\definecolor{MediumOrchid3}{rgb}{0.71,0.32,0.80}
\definecolor{MediumOrchid4}{rgb}{0.48,0.22,0.55}
\definecolor{MediumOrchid}{rgb}{0.73,0.33,0.83}
\definecolor{MediumPurple1}{rgb}{0.67,0.51,1.00}
\definecolor{MediumPurple2}{rgb}{0.62,0.47,0.93}
\definecolor{MediumPurple3}{rgb}{0.54,0.41,0.80}
\definecolor{MediumPurple4}{rgb}{0.36,0.28,0.55}
\definecolor{MediumPurple}{rgb}{0.58,0.44,0.86}
\definecolor{MediumSeaGreen}{rgb}{0.24,0.70,0.44}
\definecolor{MediumSlateBlue}{rgb}{0.48,0.41,0.93}
\definecolor{MediumSpringGreen}{rgb}{0.00,0.98,0.60}
\definecolor{MediumTurquoise}{rgb}{0.28,0.82,0.80}
\definecolor{MediumVioletRed}{rgb}{0.78,0.08,0.52}
\definecolor{MidnightBlue}{rgb}{0.10,0.10,0.44}
\definecolor{MintCream}{rgb}{0.96,1.00,0.98}
\definecolor{MistyRose1}{rgb}{1.00,0.89,0.88}
\definecolor{MistyRose2}{rgb}{0.93,0.84,0.82}
\definecolor{MistyRose3}{rgb}{0.80,0.72,0.71}
\definecolor{MistyRose4}{rgb}{0.55,0.49,0.48}
\definecolor{MistyRose}{rgb}{1.00,0.89,0.88}
\definecolor{NavajoWhite1}{rgb}{1.00,0.87,0.68}
\definecolor{NavajoWhite2}{rgb}{0.93,0.81,0.63}
\definecolor{NavajoWhite3}{rgb}{0.80,0.70,0.55}
\definecolor{NavajoWhite4}{rgb}{0.55,0.47,0.37}
\definecolor{NavajoWhite}{rgb}{1.00,0.87,0.68}
\definecolor{NavyBlue}{rgb}{0.00,0.00,0.50}
\definecolor{OldLace}{rgb}{0.99,0.96,0.90}
\definecolor{OliveDrab1}{rgb}{0.75,1.00,0.24}
\definecolor{OliveDrab2}{rgb}{0.70,0.93,0.23}
\definecolor{OliveDrab3}{rgb}{0.60,0.80,0.20}
\definecolor{OliveDrab4}{rgb}{0.41,0.55,0.13}
\definecolor{OliveDrab}{rgb}{0.42,0.56,0.14}
\definecolor{OrangeRed1}{rgb}{1.00,0.27,0.00}
\definecolor{OrangeRed2}{rgb}{0.93,0.25,0.00}
\definecolor{OrangeRed3}{rgb}{0.80,0.22,0.00}
\definecolor{OrangeRed4}{rgb}{0.55,0.15,0.00}
\definecolor{OrangeRed}{rgb}{1.00,0.27,0.00}
\definecolor{PaleGoldenrod}{rgb}{0.93,0.91,0.67}
\definecolor{PaleGreen1}{rgb}{0.60,1.00,0.60}
\definecolor{PaleGreen2}{rgb}{0.56,0.93,0.56}
\definecolor{PaleGreen3}{rgb}{0.49,0.80,0.49}
\definecolor{PaleGreen4}{rgb}{0.33,0.55,0.33}
\definecolor{PaleGreen}{rgb}{0.60,0.98,0.60}
\definecolor{PaleTurquoise1}{rgb}{0.73,1.00,1.00}
\definecolor{PaleTurquoise2}{rgb}{0.68,0.93,0.93}
\definecolor{PaleTurquoise3}{rgb}{0.59,0.80,0.80}
\definecolor{PaleTurquoise4}{rgb}{0.40,0.55,0.55}
\definecolor{PaleTurquoise}{rgb}{0.69,0.93,0.93}
\definecolor{PaleVioletRed1}{rgb}{1.00,0.51,0.67}
\definecolor{PaleVioletRed2}{rgb}{0.93,0.47,0.62}
\definecolor{PaleVioletRed3}{rgb}{0.80,0.41,0.54}
\definecolor{PaleVioletRed4}{rgb}{0.55,0.28,0.36}
\definecolor{PaleVioletRed}{rgb}{0.86,0.44,0.58}
\definecolor{PapayaWhip}{rgb}{1.00,0.94,0.84}
\definecolor{PeachPuff1}{rgb}{1.00,0.85,0.73}
\definecolor{PeachPuff2}{rgb}{0.93,0.80,0.68}
\definecolor{PeachPuff3}{rgb}{0.80,0.69,0.58}
\definecolor{PeachPuff4}{rgb}{0.55,0.47,0.40}
\definecolor{PeachPuff}{rgb}{1.00,0.85,0.73}
\definecolor{PowderBlue}{rgb}{0.69,0.88,0.90}
\definecolor{RosyBrown1}{rgb}{1.00,0.76,0.76}
\definecolor{RosyBrown2}{rgb}{0.93,0.71,0.71}
\definecolor{RosyBrown3}{rgb}{0.80,0.61,0.61}
\definecolor{RosyBrown4}{rgb}{0.55,0.41,0.41}
\definecolor{RosyBrown}{rgb}{0.74,0.56,0.56}
\definecolor{RoyalBlue1}{rgb}{0.28,0.46,1.00}
\definecolor{RoyalBlue2}{rgb}{0.26,0.43,0.93}
\definecolor{RoyalBlue3}{rgb}{0.23,0.37,0.80}
\definecolor{RoyalBlue4}{rgb}{0.15,0.25,0.55}
\definecolor{RoyalBlue}{rgb}{0.25,0.41,0.88}
\definecolor{SaddleBrown}{rgb}{0.55,0.27,0.07}
\definecolor{SandyBrown}{rgb}{0.96,0.64,0.38}
\definecolor{SeaGreen1}{rgb}{0.33,1.00,0.62}
\definecolor{SeaGreen2}{rgb}{0.31,0.93,0.58}
\definecolor{SeaGreen3}{rgb}{0.26,0.80,0.50}
\definecolor{SeaGreen4}{rgb}{0.18,0.55,0.34}
\definecolor{SeaGreen}{rgb}{0.18,0.55,0.34}
\definecolor{SkyBlue1}{rgb}{0.53,0.81,1.00}
\definecolor{SkyBlue2}{rgb}{0.49,0.75,0.93}
\definecolor{SkyBlue3}{rgb}{0.42,0.65,0.80}
\definecolor{SkyBlue4}{rgb}{0.29,0.44,0.55}
\definecolor{SkyBlue}{rgb}{0.53,0.81,0.92}
\definecolor{SlateBlue1}{rgb}{0.51,0.44,1.00}
\definecolor{SlateBlue2}{rgb}{0.48,0.40,0.93}
\definecolor{SlateBlue3}{rgb}{0.41,0.35,0.80}
\definecolor{SlateBlue4}{rgb}{0.28,0.24,0.55}
\definecolor{SlateBlue}{rgb}{0.42,0.35,0.80}
\definecolor{SlateGray1}{rgb}{0.78,0.89,1.00}
\definecolor{SlateGray2}{rgb}{0.73,0.83,0.93}
\definecolor{SlateGray3}{rgb}{0.62,0.71,0.80}
\definecolor{SlateGray4}{rgb}{0.42,0.48,0.55}
\definecolor{SlateGray}{rgb}{0.44,0.50,0.56}
\definecolor{SlateGrey}{rgb}{0.44,0.50,0.56}
\definecolor{SpringGreen1}{rgb}{0.00,1.00,0.50}
\definecolor{SpringGreen2}{rgb}{0.00,0.93,0.46}
\definecolor{SpringGreen3}{rgb}{0.00,0.80,0.40}
\definecolor{SpringGreen4}{rgb}{0.00,0.55,0.27}
\definecolor{SpringGreen}{rgb}{0.00,1.00,0.50}
\definecolor{SteelBlue1}{rgb}{0.39,0.72,1.00}
\definecolor{SteelBlue2}{rgb}{0.36,0.67,0.93}
\definecolor{SteelBlue3}{rgb}{0.31,0.58,0.80}
\definecolor{SteelBlue4}{rgb}{0.21,0.39,0.55}
\definecolor{SteelBlue}{rgb}{0.27,0.51,0.71}
\definecolor{VioletRed1}{rgb}{1.00,0.24,0.59}
\definecolor{VioletRed2}{rgb}{0.93,0.23,0.55}
\definecolor{VioletRed3}{rgb}{0.80,0.20,0.47}
\definecolor{VioletRed4}{rgb}{0.55,0.13,0.32}
\definecolor{VioletRed}{rgb}{0.82,0.13,0.56}
\definecolor{WhiteSmoke}{rgb}{0.96,0.96,0.96}
\definecolor{YellowGreen}{rgb}{0.60,0.80,0.20}
\definecolor{aliceblue}{rgb}{0.94,0.97,1.00}
\definecolor{antiquewhite}{rgb}{0.98,0.92,0.84}
\definecolor{aquamarine1}{rgb}{0.50,1.00,0.83}
\definecolor{aquamarine2}{rgb}{0.46,0.93,0.78}
\definecolor{aquamarine3}{rgb}{0.40,0.80,0.67}
\definecolor{aquamarine4}{rgb}{0.27,0.55,0.45}
\definecolor{aquamarine}{rgb}{0.50,1.00,0.83}
\definecolor{azure1}{rgb}{0.94,1.00,1.00}
\definecolor{azure2}{rgb}{0.88,0.93,0.93}
\definecolor{azure3}{rgb}{0.76,0.80,0.80}
\definecolor{azure4}{rgb}{0.51,0.55,0.55}
\definecolor{azure}{rgb}{0.94,1.00,1.00}
\definecolor{beige}{rgb}{0.96,0.96,0.86}
\definecolor{bisque1}{rgb}{1.00,0.89,0.77}
\definecolor{bisque2}{rgb}{0.93,0.84,0.72}
\definecolor{bisque3}{rgb}{0.80,0.72,0.62}
\definecolor{bisque4}{rgb}{0.55,0.49,0.42}
\definecolor{bisque}{rgb}{1.00,0.89,0.77}
\definecolor{black}{rgb}{0.00,0.00,0.00}
\definecolor{blanchedalmond}{rgb}{1.00,0.92,0.80}
\definecolor{blue1}{rgb}{0.00,0.00,1.00}
\definecolor{blue2}{rgb}{0.00,0.00,0.93}
\definecolor{blue3}{rgb}{0.00,0.00,0.80}
\definecolor{blue4}{rgb}{0.00,0.00,0.55}
\definecolor{blueviolet}{rgb}{0.54,0.17,0.89}
\definecolor{blue}{rgb}{0.00,0.00,1.00}
\definecolor{brown1}{rgb}{1.00,0.25,0.25}
\definecolor{brown2}{rgb}{0.93,0.23,0.23}
\definecolor{brown3}{rgb}{0.80,0.20,0.20}
\definecolor{brown4}{rgb}{0.55,0.14,0.14}
\definecolor{brown}{rgb}{0.65,0.16,0.16}
\definecolor{burlywood1}{rgb}{1.00,0.83,0.61}
\definecolor{burlywood2}{rgb}{0.93,0.77,0.57}
\definecolor{burlywood3}{rgb}{0.80,0.67,0.49}
\definecolor{burlywood4}{rgb}{0.55,0.45,0.33}
\definecolor{burlywood}{rgb}{0.87,0.72,0.53}
\definecolor{cadetblue}{rgb}{0.37,0.62,0.63}
\definecolor{chartreuse1}{rgb}{0.50,1.00,0.00}
\definecolor{chartreuse2}{rgb}{0.46,0.93,0.00}
\definecolor{chartreuse3}{rgb}{0.40,0.80,0.00}
\definecolor{chartreuse4}{rgb}{0.27,0.55,0.00}
\definecolor{chartreuse}{rgb}{0.50,1.00,0.00}
\definecolor{chocolate1}{rgb}{1.00,0.50,0.14}
\definecolor{chocolate2}{rgb}{0.93,0.46,0.13}
\definecolor{chocolate3}{rgb}{0.80,0.40,0.11}
\definecolor{chocolate4}{rgb}{0.55,0.27,0.07}
\definecolor{chocolate}{rgb}{0.82,0.41,0.12}
\definecolor{coral1}{rgb}{1.00,0.45,0.34}
\definecolor{coral2}{rgb}{0.93,0.42,0.31}
\definecolor{coral3}{rgb}{0.80,0.36,0.27}
\definecolor{coral4}{rgb}{0.55,0.24,0.18}
\definecolor{coral}{rgb}{1.00,0.50,0.31}
\definecolor{cornflowerblue}{rgb}{0.39,0.58,0.93}
\definecolor{cornsilk1}{rgb}{1.00,0.97,0.86}
\definecolor{cornsilk2}{rgb}{0.93,0.91,0.80}
\definecolor{cornsilk3}{rgb}{0.80,0.78,0.69}
\definecolor{cornsilk4}{rgb}{0.55,0.53,0.47}
\definecolor{cornsilk}{rgb}{1.00,0.97,0.86}
\definecolor{cyan1}{rgb}{0.00,1.00,1.00}
\definecolor{cyan2}{rgb}{0.00,0.93,0.93}
\definecolor{cyan3}{rgb}{0.00,0.80,0.80}
\definecolor{cyan4}{rgb}{0.00,0.55,0.55}
\definecolor{cyan}{rgb}{0.00,1.00,1.00}
\definecolor{darkblue}{rgb}{0.00,0.00,0.55}
\definecolor{darkcyan}{rgb}{0.00,0.55,0.55}
\definecolor{darkgoldenrod}{rgb}{0.72,0.53,0.04}
\definecolor{darkgray}{rgb}{0.66,0.66,0.66}
\definecolor{darkgreen}{rgb}{0.00,0.39,0.00}
\definecolor{darkgrey}{rgb}{0.66,0.66,0.66}
\definecolor{darkkhaki}{rgb}{0.74,0.72,0.42}
\definecolor{darkmagenta}{rgb}{0.55,0.00,0.55}
\definecolor{darkolive}{rgb}{0.33,0.42,0.18}
\definecolor{darkorange}{rgb}{1.00,0.55,0.00}
\definecolor{darkorchid}{rgb}{0.60,0.20,0.80}
\definecolor{darkred}{rgb}{0.55,0.00,0.00}
\definecolor{darksalmon}{rgb}{0.91,0.59,0.48}
\definecolor{darksea}{rgb}{0.56,0.74,0.56}
\definecolor{darkslate}{rgb}{0.18,0.31,0.31}
\definecolor{darkslate}{rgb}{0.18,0.31,0.31}
\definecolor{darkslate}{rgb}{0.28,0.24,0.55}
\definecolor{darkturquoise}{rgb}{0.00,0.81,0.82}
\definecolor{darkviolet}{rgb}{0.58,0.00,0.83}
\definecolor{deeppink}{rgb}{1.00,0.08,0.58}
\definecolor{deepsky}{rgb}{0.00,0.75,1.00}
\definecolor{dimgray}{rgb}{0.41,0.41,0.41}
\definecolor{dimgrey}{rgb}{0.41,0.41,0.41}
\definecolor{dodgerblue}{rgb}{0.12,0.56,1.00}
\definecolor{firebrick1}{rgb}{1.00,0.19,0.19}
\definecolor{firebrick2}{rgb}{0.93,0.17,0.17}
\definecolor{firebrick3}{rgb}{0.80,0.15,0.15}
\definecolor{firebrick4}{rgb}{0.55,0.10,0.10}
\definecolor{firebrick}{rgb}{0.70,0.13,0.13}
\definecolor{floralwhite}{rgb}{1.00,0.98,0.94}
\definecolor{forestgreen}{rgb}{0.13,0.55,0.13}
\definecolor{gainsboro}{rgb}{0.86,0.86,0.86}
\definecolor{ghostwhite}{rgb}{0.97,0.97,1.00}
\definecolor{gold1}{rgb}{1.00,0.84,0.00}
\definecolor{gold2}{rgb}{0.93,0.79,0.00}
\definecolor{gold3}{rgb}{0.80,0.68,0.00}
\definecolor{gold4}{rgb}{0.55,0.46,0.00}
\definecolor{goldenrod1}{rgb}{1.00,0.76,0.15}
\definecolor{goldenrod2}{rgb}{0.93,0.71,0.13}
\definecolor{goldenrod3}{rgb}{0.80,0.61,0.11}
\definecolor{goldenrod4}{rgb}{0.55,0.41,0.08}
\definecolor{goldenrod}{rgb}{0.85,0.65,0.13}
\definecolor{gold}{rgb}{1.00,0.84,0.00}
\definecolor{gray0}{rgb}{0.00,0.00,0.00}
\definecolor{gray100}{rgb}{1.00,1.00,1.00}
\definecolor{gray10}{rgb}{0.10,0.10,0.10}
\definecolor{gray11}{rgb}{0.11,0.11,0.11}
\definecolor{gray12}{rgb}{0.12,0.12,0.12}
\definecolor{gray13}{rgb}{0.13,0.13,0.13}
\definecolor{gray14}{rgb}{0.14,0.14,0.14}
\definecolor{gray15}{rgb}{0.15,0.15,0.15}
\definecolor{gray16}{rgb}{0.16,0.16,0.16}
\definecolor{gray17}{rgb}{0.17,0.17,0.17}
\definecolor{gray18}{rgb}{0.18,0.18,0.18}
\definecolor{gray19}{rgb}{0.19,0.19,0.19}
\definecolor{gray1}{rgb}{0.01,0.01,0.01}
\definecolor{gray20}{rgb}{0.20,0.20,0.20}
\definecolor{gray21}{rgb}{0.21,0.21,0.21}
\definecolor{gray22}{rgb}{0.22,0.22,0.22}
\definecolor{gray23}{rgb}{0.23,0.23,0.23}
\definecolor{gray24}{rgb}{0.24,0.24,0.24}
\definecolor{gray25}{rgb}{0.25,0.25,0.25}
\definecolor{gray26}{rgb}{0.26,0.26,0.26}
\definecolor{gray27}{rgb}{0.27,0.27,0.27}
\definecolor{gray28}{rgb}{0.28,0.28,0.28}
\definecolor{gray29}{rgb}{0.29,0.29,0.29}
\definecolor{gray2}{rgb}{0.02,0.02,0.02}
\definecolor{gray30}{rgb}{0.30,0.30,0.30}
\definecolor{gray31}{rgb}{0.31,0.31,0.31}
\definecolor{gray32}{rgb}{0.32,0.32,0.32}
\definecolor{gray33}{rgb}{0.33,0.33,0.33}
\definecolor{gray34}{rgb}{0.34,0.34,0.34}
\definecolor{gray35}{rgb}{0.35,0.35,0.35}
\definecolor{gray36}{rgb}{0.36,0.36,0.36}
\definecolor{gray37}{rgb}{0.37,0.37,0.37}
\definecolor{gray38}{rgb}{0.38,0.38,0.38}
\definecolor{gray39}{rgb}{0.39,0.39,0.39}
\definecolor{gray3}{rgb}{0.03,0.03,0.03}
\definecolor{gray40}{rgb}{0.40,0.40,0.40}
\definecolor{gray41}{rgb}{0.41,0.41,0.41}
\definecolor{gray42}{rgb}{0.42,0.42,0.42}
\definecolor{gray43}{rgb}{0.43,0.43,0.43}
\definecolor{gray44}{rgb}{0.44,0.44,0.44}
\definecolor{gray45}{rgb}{0.45,0.45,0.45}
\definecolor{gray46}{rgb}{0.46,0.46,0.46}
\definecolor{gray47}{rgb}{0.47,0.47,0.47}
\definecolor{gray48}{rgb}{0.48,0.48,0.48}
\definecolor{gray49}{rgb}{0.49,0.49,0.49}
\definecolor{gray4}{rgb}{0.04,0.04,0.04}
\definecolor{gray50}{rgb}{0.50,0.50,0.50}
\definecolor{gray51}{rgb}{0.51,0.51,0.51}
\definecolor{gray52}{rgb}{0.52,0.52,0.52}
\definecolor{gray53}{rgb}{0.53,0.53,0.53}
\definecolor{gray54}{rgb}{0.54,0.54,0.54}
\definecolor{gray55}{rgb}{0.55,0.55,0.55}
\definecolor{gray56}{rgb}{0.56,0.56,0.56}
\definecolor{gray57}{rgb}{0.57,0.57,0.57}
\definecolor{gray58}{rgb}{0.58,0.58,0.58}
\definecolor{gray59}{rgb}{0.59,0.59,0.59}
\definecolor{gray5}{rgb}{0.05,0.05,0.05}
\definecolor{gray60}{rgb}{0.60,0.60,0.60}
\definecolor{gray61}{rgb}{0.61,0.61,0.61}
\definecolor{gray62}{rgb}{0.62,0.62,0.62}
\definecolor{gray63}{rgb}{0.63,0.63,0.63}
\definecolor{gray64}{rgb}{0.64,0.64,0.64}
\definecolor{gray65}{rgb}{0.65,0.65,0.65}
\definecolor{gray66}{rgb}{0.66,0.66,0.66}
\definecolor{gray67}{rgb}{0.67,0.67,0.67}
\definecolor{gray68}{rgb}{0.68,0.68,0.68}
\definecolor{gray69}{rgb}{0.69,0.69,0.69}
\definecolor{gray6}{rgb}{0.06,0.06,0.06}
\definecolor{gray70}{rgb}{0.70,0.70,0.70}
\definecolor{gray71}{rgb}{0.71,0.71,0.71}
\definecolor{gray72}{rgb}{0.72,0.72,0.72}
\definecolor{gray73}{rgb}{0.73,0.73,0.73}
\definecolor{gray74}{rgb}{0.74,0.74,0.74}
\definecolor{gray75}{rgb}{0.75,0.75,0.75}
\definecolor{gray76}{rgb}{0.76,0.76,0.76}
\definecolor{gray77}{rgb}{0.77,0.77,0.77}
\definecolor{gray78}{rgb}{0.78,0.78,0.78}
\definecolor{gray79}{rgb}{0.79,0.79,0.79}
\definecolor{gray7}{rgb}{0.07,0.07,0.07}
\definecolor{gray80}{rgb}{0.80,0.80,0.80}
\definecolor{gray81}{rgb}{0.81,0.81,0.81}
\definecolor{gray82}{rgb}{0.82,0.82,0.82}
\definecolor{gray83}{rgb}{0.83,0.83,0.83}
\definecolor{gray84}{rgb}{0.84,0.84,0.84}
\definecolor{gray85}{rgb}{0.85,0.85,0.85}
\definecolor{gray86}{rgb}{0.86,0.86,0.86}
\definecolor{gray87}{rgb}{0.87,0.87,0.87}
\definecolor{gray88}{rgb}{0.88,0.88,0.88}
\definecolor{gray89}{rgb}{0.89,0.89,0.89}
\definecolor{gray8}{rgb}{0.08,0.08,0.08}
\definecolor{gray90}{rgb}{0.90,0.90,0.90}
\definecolor{gray91}{rgb}{0.91,0.91,0.91}
\definecolor{gray92}{rgb}{0.92,0.92,0.92}
\definecolor{gray93}{rgb}{0.93,0.93,0.93}
\definecolor{gray94}{rgb}{0.94,0.94,0.94}
\definecolor{gray95}{rgb}{0.95,0.95,0.95}
\definecolor{gray96}{rgb}{0.96,0.96,0.96}
\definecolor{gray97}{rgb}{0.97,0.97,0.97}
\definecolor{gray98}{rgb}{0.98,0.98,0.98}
\definecolor{gray99}{rgb}{0.99,0.99,0.99}
\definecolor{gray9}{rgb}{0.09,0.09,0.09}
\definecolor{gray}{rgb}{0.75,0.75,0.75}
\definecolor{green1}{rgb}{0.00,1.00,0.00}
\definecolor{green2}{rgb}{0.00,0.93,0.00}
\definecolor{green3}{rgb}{0.00,0.80,0.00}
\definecolor{green4}{rgb}{0.00,0.55,0.00}
\definecolor{greenyellow}{rgb}{0.68,1.00,0.18}
\definecolor{green}{rgb}{0.00,1.00,0.00}
\definecolor{grey0}{rgb}{0.00,0.00,0.00}
\definecolor{grey100}{rgb}{1.00,1.00,1.00}
\definecolor{grey10}{rgb}{0.10,0.10,0.10}
\definecolor{grey11}{rgb}{0.11,0.11,0.11}
\definecolor{grey12}{rgb}{0.12,0.12,0.12}
\definecolor{grey13}{rgb}{0.13,0.13,0.13}
\definecolor{grey14}{rgb}{0.14,0.14,0.14}
\definecolor{grey15}{rgb}{0.15,0.15,0.15}
\definecolor{grey16}{rgb}{0.16,0.16,0.16}
\definecolor{grey17}{rgb}{0.17,0.17,0.17}
\definecolor{grey18}{rgb}{0.18,0.18,0.18}
\definecolor{grey19}{rgb}{0.19,0.19,0.19}
\definecolor{grey1}{rgb}{0.01,0.01,0.01}
\definecolor{grey20}{rgb}{0.20,0.20,0.20}
\definecolor{grey21}{rgb}{0.21,0.21,0.21}
\definecolor{grey22}{rgb}{0.22,0.22,0.22}
\definecolor{grey23}{rgb}{0.23,0.23,0.23}
\definecolor{grey24}{rgb}{0.24,0.24,0.24}
\definecolor{grey25}{rgb}{0.25,0.25,0.25}
\definecolor{grey26}{rgb}{0.26,0.26,0.26}
\definecolor{grey27}{rgb}{0.27,0.27,0.27}
\definecolor{grey28}{rgb}{0.28,0.28,0.28}
\definecolor{grey29}{rgb}{0.29,0.29,0.29}
\definecolor{grey2}{rgb}{0.02,0.02,0.02}
\definecolor{grey30}{rgb}{0.30,0.30,0.30}
\definecolor{grey31}{rgb}{0.31,0.31,0.31}
\definecolor{grey32}{rgb}{0.32,0.32,0.32}
\definecolor{grey33}{rgb}{0.33,0.33,0.33}
\definecolor{grey34}{rgb}{0.34,0.34,0.34}
\definecolor{grey35}{rgb}{0.35,0.35,0.35}
\definecolor{grey36}{rgb}{0.36,0.36,0.36}
\definecolor{grey37}{rgb}{0.37,0.37,0.37}
\definecolor{grey38}{rgb}{0.38,0.38,0.38}
\definecolor{grey39}{rgb}{0.39,0.39,0.39}
\definecolor{grey3}{rgb}{0.03,0.03,0.03}
\definecolor{grey40}{rgb}{0.40,0.40,0.40}
\definecolor{grey41}{rgb}{0.41,0.41,0.41}
\definecolor{grey42}{rgb}{0.42,0.42,0.42}
\definecolor{grey43}{rgb}{0.43,0.43,0.43}
\definecolor{grey44}{rgb}{0.44,0.44,0.44}
\definecolor{grey45}{rgb}{0.45,0.45,0.45}
\definecolor{grey46}{rgb}{0.46,0.46,0.46}
\definecolor{grey47}{rgb}{0.47,0.47,0.47}
\definecolor{grey48}{rgb}{0.48,0.48,0.48}
\definecolor{grey49}{rgb}{0.49,0.49,0.49}
\definecolor{grey4}{rgb}{0.04,0.04,0.04}
\definecolor{grey50}{rgb}{0.50,0.50,0.50}
\definecolor{grey51}{rgb}{0.51,0.51,0.51}
\definecolor{grey52}{rgb}{0.52,0.52,0.52}
\definecolor{grey53}{rgb}{0.53,0.53,0.53}
\definecolor{grey54}{rgb}{0.54,0.54,0.54}
\definecolor{grey55}{rgb}{0.55,0.55,0.55}
\definecolor{grey56}{rgb}{0.56,0.56,0.56}
\definecolor{grey57}{rgb}{0.57,0.57,0.57}
\definecolor{grey58}{rgb}{0.58,0.58,0.58}
\definecolor{grey59}{rgb}{0.59,0.59,0.59}
\definecolor{grey5}{rgb}{0.05,0.05,0.05}
\definecolor{grey60}{rgb}{0.60,0.60,0.60}
\definecolor{grey61}{rgb}{0.61,0.61,0.61}
\definecolor{grey62}{rgb}{0.62,0.62,0.62}
\definecolor{grey63}{rgb}{0.63,0.63,0.63}
\definecolor{grey64}{rgb}{0.64,0.64,0.64}
\definecolor{grey65}{rgb}{0.65,0.65,0.65}
\definecolor{grey66}{rgb}{0.66,0.66,0.66}
\definecolor{grey67}{rgb}{0.67,0.67,0.67}
\definecolor{grey68}{rgb}{0.68,0.68,0.68}
\definecolor{grey69}{rgb}{0.69,0.69,0.69}
\definecolor{grey6}{rgb}{0.06,0.06,0.06}
\definecolor{grey70}{rgb}{0.70,0.70,0.70}
\definecolor{grey71}{rgb}{0.71,0.71,0.71}
\definecolor{grey72}{rgb}{0.72,0.72,0.72}
\definecolor{grey73}{rgb}{0.73,0.73,0.73}
\definecolor{grey74}{rgb}{0.74,0.74,0.74}
\definecolor{grey75}{rgb}{0.75,0.75,0.75}
\definecolor{grey76}{rgb}{0.76,0.76,0.76}
\definecolor{grey77}{rgb}{0.77,0.77,0.77}
\definecolor{grey78}{rgb}{0.78,0.78,0.78}
\definecolor{grey79}{rgb}{0.79,0.79,0.79}
\definecolor{grey7}{rgb}{0.07,0.07,0.07}
\definecolor{grey80}{rgb}{0.80,0.80,0.80}
\definecolor{grey81}{rgb}{0.81,0.81,0.81}
\definecolor{grey82}{rgb}{0.82,0.82,0.82}
\definecolor{grey83}{rgb}{0.83,0.83,0.83}
\definecolor{grey84}{rgb}{0.84,0.84,0.84}
\definecolor{grey85}{rgb}{0.85,0.85,0.85}
\definecolor{grey86}{rgb}{0.86,0.86,0.86}
\definecolor{grey87}{rgb}{0.87,0.87,0.87}
\definecolor{grey88}{rgb}{0.88,0.88,0.88}
\definecolor{grey89}{rgb}{0.89,0.89,0.89}
\definecolor{grey8}{rgb}{0.08,0.08,0.08}
\definecolor{grey90}{rgb}{0.90,0.90,0.90}
\definecolor{grey91}{rgb}{0.91,0.91,0.91}
\definecolor{grey92}{rgb}{0.92,0.92,0.92}
\definecolor{grey93}{rgb}{0.93,0.93,0.93}
\definecolor{grey94}{rgb}{0.94,0.94,0.94}
\definecolor{grey95}{rgb}{0.95,0.95,0.95}
\definecolor{grey96}{rgb}{0.96,0.96,0.96}
\definecolor{grey97}{rgb}{0.97,0.97,0.97}
\definecolor{grey98}{rgb}{0.98,0.98,0.98}
\definecolor{grey99}{rgb}{0.99,0.99,0.99}
\definecolor{grey9}{rgb}{0.09,0.09,0.09}
\definecolor{grey}{rgb}{0.75,0.75,0.75}
\definecolor{honeydew1}{rgb}{0.94,1.00,0.94}
\definecolor{honeydew2}{rgb}{0.88,0.93,0.88}
\definecolor{honeydew3}{rgb}{0.76,0.80,0.76}
\definecolor{honeydew4}{rgb}{0.51,0.55,0.51}
\definecolor{honeydew}{rgb}{0.94,1.00,0.94}
\definecolor{hotpink}{rgb}{1.00,0.41,0.71}
\definecolor{indianred}{rgb}{0.80,0.36,0.36}
\definecolor{ivory1}{rgb}{1.00,1.00,0.94}
\definecolor{ivory2}{rgb}{0.93,0.93,0.88}
\definecolor{ivory3}{rgb}{0.80,0.80,0.76}
\definecolor{ivory4}{rgb}{0.55,0.55,0.51}
\definecolor{ivory}{rgb}{1.00,1.00,0.94}
\definecolor{khaki1}{rgb}{1.00,0.96,0.56}
\definecolor{khaki2}{rgb}{0.93,0.90,0.52}
\definecolor{khaki3}{rgb}{0.80,0.78,0.45}
\definecolor{khaki4}{rgb}{0.55,0.53,0.31}
\definecolor{khaki}{rgb}{0.94,0.90,0.55}
\definecolor{lavenderblush}{rgb}{1.00,0.94,0.96}
\definecolor{lavender}{rgb}{0.90,0.90,0.98}
\definecolor{lawngreen}{rgb}{0.49,0.99,0.00}
\definecolor{lemonchiffon}{rgb}{1.00,0.98,0.80}
\definecolor{lightblue}{rgb}{0.68,0.85,0.90}
\definecolor{lightcoral}{rgb}{0.94,0.50,0.50}
\definecolor{lightcyan}{rgb}{0.88,1.00,1.00}
\definecolor{lightgoldenrod}{rgb}{0.93,0.87,0.51}
\definecolor{lightgoldenrod}{rgb}{0.98,0.98,0.82}
\definecolor{lightgray}{rgb}{0.83,0.83,0.83}
\definecolor{lightgreen}{rgb}{0.56,0.93,0.56}
\definecolor{lightgrey}{rgb}{0.83,0.83,0.83}
\definecolor{lightpink}{rgb}{1.00,0.71,0.76}
\definecolor{lightsalmon}{rgb}{1.00,0.63,0.48}
\definecolor{lightsea}{rgb}{0.13,0.70,0.67}
\definecolor{lightsky}{rgb}{0.53,0.81,0.98}
\definecolor{lightslate}{rgb}{0.47,0.53,0.60}
\definecolor{lightslate}{rgb}{0.47,0.53,0.60}
\definecolor{lightslate}{rgb}{0.52,0.44,1.00}
\definecolor{lightsteel}{rgb}{0.69,0.77,0.87}
\definecolor{lightyellow}{rgb}{1.00,1.00,0.88}
\definecolor{limegreen}{rgb}{0.20,0.80,0.20}
\definecolor{linen}{rgb}{0.98,0.94,0.90}
\definecolor{magenta1}{rgb}{1.00,0.00,1.00}
\definecolor{magenta2}{rgb}{0.93,0.00,0.93}
\definecolor{magenta3}{rgb}{0.80,0.00,0.80}
\definecolor{magenta4}{rgb}{0.55,0.00,0.55}
\definecolor{magenta}{rgb}{1.00,0.00,1.00}
\definecolor{maroon1}{rgb}{1.00,0.20,0.70}
\definecolor{maroon2}{rgb}{0.93,0.19,0.65}
\definecolor{maroon3}{rgb}{0.80,0.16,0.56}
\definecolor{maroon4}{rgb}{0.55,0.11,0.38}
\definecolor{maroon}{rgb}{0.69,0.19,0.38}
\definecolor{mediumaquamarine}{rgb}{0.40,0.80,0.67}
\definecolor{mediumblue}{rgb}{0.00,0.00,0.80}
\definecolor{mediumorchid}{rgb}{0.73,0.33,0.83}
\definecolor{mediumpurple}{rgb}{0.58,0.44,0.86}
\definecolor{mediumsea}{rgb}{0.24,0.70,0.44}
\definecolor{mediumslate}{rgb}{0.48,0.41,0.93}
\definecolor{mediumspring}{rgb}{0.00,0.98,0.60}
\definecolor{mediumturquoise}{rgb}{0.28,0.82,0.80}
\definecolor{mediumviolet}{rgb}{0.78,0.08,0.52}
\definecolor{midnightblue}{rgb}{0.10,0.10,0.44}
\definecolor{mintcream}{rgb}{0.96,1.00,0.98}
\definecolor{mistyrose}{rgb}{1.00,0.89,0.88}
\definecolor{moccasin}{rgb}{1.00,0.89,0.71}
\definecolor{navajowhite}{rgb}{1.00,0.87,0.68}
\definecolor{navyblue}{rgb}{0.00,0.00,0.50}
\definecolor{navy}{rgb}{0.00,0.00,0.50}
\definecolor{oldlace}{rgb}{0.99,0.96,0.90}
\definecolor{olivedrab}{rgb}{0.42,0.56,0.14}
\definecolor{orange1}{rgb}{1.00,0.65,0.00}
\definecolor{orange2}{rgb}{0.93,0.60,0.00}
\definecolor{orange3}{rgb}{0.80,0.52,0.00}
\definecolor{orange4}{rgb}{0.55,0.35,0.00}
\definecolor{orangered}{rgb}{1.00,0.27,0.00}
\definecolor{orange}{rgb}{1.00,0.65,0.00}
\definecolor{orchid1}{rgb}{1.00,0.51,0.98}
\definecolor{orchid2}{rgb}{0.93,0.48,0.91}
\definecolor{orchid3}{rgb}{0.80,0.41,0.79}
\definecolor{orchid4}{rgb}{0.55,0.28,0.54}
\definecolor{orchid}{rgb}{0.85,0.44,0.84}
\definecolor{palegoldenrod}{rgb}{0.93,0.91,0.67}
\definecolor{palegreen}{rgb}{0.60,0.98,0.60}
\definecolor{paleturquoise}{rgb}{0.69,0.93,0.93}
\definecolor{paleviolet}{rgb}{0.86,0.44,0.58}
\definecolor{papayawhip}{rgb}{1.00,0.94,0.84}
\definecolor{peachpuff}{rgb}{1.00,0.85,0.73}
\definecolor{peru}{rgb}{0.80,0.52,0.25}
\definecolor{pink1}{rgb}{1.00,0.71,0.77}
\definecolor{pink2}{rgb}{0.93,0.66,0.72}
\definecolor{pink3}{rgb}{0.80,0.57,0.62}
\definecolor{pink4}{rgb}{0.55,0.39,0.42}
\definecolor{pink}{rgb}{1.00,0.75,0.80}
\definecolor{plum1}{rgb}{1.00,0.73,1.00}
\definecolor{plum2}{rgb}{0.93,0.68,0.93}
\definecolor{plum3}{rgb}{0.80,0.59,0.80}
\definecolor{plum4}{rgb}{0.55,0.40,0.55}
\definecolor{plum}{rgb}{0.87,0.63,0.87}
\definecolor{powderblue}{rgb}{0.69,0.88,0.90}
\definecolor{purple1}{rgb}{0.61,0.19,1.00}
\definecolor{purple2}{rgb}{0.57,0.17,0.93}
\definecolor{purple3}{rgb}{0.49,0.15,0.80}
\definecolor{purple4}{rgb}{0.33,0.10,0.55}
\definecolor{purple}{rgb}{0.63,0.13,0.94}
\definecolor{red1}{rgb}{1.00,0.00,0.00}
\definecolor{red2}{rgb}{0.93,0.00,0.00}
\definecolor{red3}{rgb}{0.80,0.00,0.00}
\definecolor{red4}{rgb}{0.55,0.00,0.00}
\definecolor{red}{rgb}{1.00,0.00,0.00}
\definecolor{rosybrown}{rgb}{0.74,0.56,0.56}
\definecolor{royalblue}{rgb}{0.25,0.41,0.88}
\definecolor{saddlebrown}{rgb}{0.55,0.27,0.07}
\definecolor{salmon1}{rgb}{1.00,0.55,0.41}
\definecolor{salmon2}{rgb}{0.93,0.51,0.38}
\definecolor{salmon3}{rgb}{0.80,0.44,0.33}
\definecolor{salmon4}{rgb}{0.55,0.30,0.22}
\definecolor{salmon}{rgb}{0.98,0.50,0.45}
\definecolor{sandybrown}{rgb}{0.96,0.64,0.38}
\definecolor{seagreen}{rgb}{0.18,0.55,0.34}
\definecolor{seashell1}{rgb}{1.00,0.96,0.93}
\definecolor{seashell2}{rgb}{0.93,0.90,0.87}
\definecolor{seashell3}{rgb}{0.80,0.77,0.75}
\definecolor{seashell4}{rgb}{0.55,0.53,0.51}
\definecolor{seashell}{rgb}{1.00,0.96,0.93}
\definecolor{sienna1}{rgb}{1.00,0.51,0.28}
\definecolor{sienna2}{rgb}{0.93,0.47,0.26}
\definecolor{sienna3}{rgb}{0.80,0.41,0.22}
\definecolor{sienna4}{rgb}{0.55,0.28,0.15}
\definecolor{sienna}{rgb}{0.63,0.32,0.18}
\definecolor{skyblue}{rgb}{0.53,0.81,0.92}
\definecolor{slateblue}{rgb}{0.42,0.35,0.80}
\definecolor{slategray}{rgb}{0.44,0.50,0.56}
\definecolor{slategrey}{rgb}{0.44,0.50,0.56}
\definecolor{snow1}{rgb}{1.00,0.98,0.98}
\definecolor{snow2}{rgb}{0.93,0.91,0.91}
\definecolor{snow3}{rgb}{0.80,0.79,0.79}
\definecolor{snow4}{rgb}{0.55,0.54,0.54}
\definecolor{snow}{rgb}{1.00,0.98,0.98}
\definecolor{springgreen}{rgb}{0.00,1.00,0.50}
\definecolor{steelblue}{rgb}{0.27,0.51,0.71}
\definecolor{tan1}{rgb}{1.00,0.65,0.31}
\definecolor{tan2}{rgb}{0.93,0.60,0.29}
\definecolor{tan3}{rgb}{0.80,0.52,0.25}
\definecolor{tan4}{rgb}{0.55,0.35,0.17}
\definecolor{tan}{rgb}{0.82,0.71,0.55}
\definecolor{thistle1}{rgb}{1.00,0.88,1.00}
\definecolor{thistle2}{rgb}{0.93,0.82,0.93}
\definecolor{thistle3}{rgb}{0.80,0.71,0.80}
\definecolor{thistle4}{rgb}{0.55,0.48,0.55}
\definecolor{thistle}{rgb}{0.85,0.75,0.85}
\definecolor{tomato1}{rgb}{1.00,0.39,0.28}
\definecolor{tomato2}{rgb}{0.93,0.36,0.26}
\definecolor{tomato3}{rgb}{0.80,0.31,0.22}
\definecolor{tomato4}{rgb}{0.55,0.21,0.15}
\definecolor{tomato}{rgb}{1.00,0.39,0.28}
\definecolor{turquoise1}{rgb}{0.00,0.96,1.00}
\definecolor{turquoise2}{rgb}{0.00,0.90,0.93}
\definecolor{turquoise3}{rgb}{0.00,0.77,0.80}
\definecolor{turquoise4}{rgb}{0.00,0.53,0.55}
\definecolor{turquoise}{rgb}{0.25,0.88,0.82}
\definecolor{violetred}{rgb}{0.82,0.13,0.56}
\definecolor{violet}{rgb}{0.93,0.51,0.93}
\definecolor{wheat1}{rgb}{1.00,0.91,0.73}
\definecolor{wheat2}{rgb}{0.93,0.85,0.68}
\definecolor{wheat3}{rgb}{0.80,0.73,0.59}
\definecolor{wheat4}{rgb}{0.55,0.49,0.40}
\definecolor{wheat}{rgb}{0.96,0.87,0.70}
\definecolor{whitesmoke}{rgb}{0.96,0.96,0.96}
\definecolor{white}{rgb}{1.00,1.00,1.00}
\definecolor{yellow1}{rgb}{1.00,1.00,0.00}
\definecolor{yellow2}{rgb}{0.93,0.93,0.00}
\definecolor{yellow3}{rgb}{0.80,0.80,0.00}
\definecolor{yellow4}{rgb}{0.55,0.55,0.00}
\definecolor{yellowgreen}{rgb}{0.60,0.80,0.20}
\definecolor{yellow}{rgb}{1.00,1.00,0.00}